\documentclass[11pt]{article}
\usepackage{amsmath}
\usepackage{amsthm}
\usepackage{amssymb}
\usepackage{graphicx}
\usepackage{mathtools}
\usepackage[margin=2.5cm]{geometry}
\usepackage{setspace}
\usepackage{epstopdf}
\usepackage{subfig}
\usepackage{empheq}
\usepackage{algorithm}
\usepackage{algorithmicx}
\usepackage{algpseudocode}
\usepackage{multirow}
\usepackage{url}
\usepackage[title]{appendix}
\usepackage{enumitem}
\usepackage[svgnames]{xcolor}

\usepackage{bbm}
\usepackage{hyperref}
\hypersetup{
    colorlinks=true, 
    citecolor=black,
     linkcolor=blue,
     urlcolor=blue,
    linktoc=all,     
    linkcolor=black,  
}

\makeatletter

\theoremstyle{plain}
\newtheorem{thm}{\protect\theoremname}
  \theoremstyle{definition}

  \theoremstyle{plain}
  \newtheorem{lem}[thm]{\protect\lemmaname}
  
  \newtheorem{cor}[thm]{\protect\corollaryname}

  \providecommand{\propname}{Proposition}
  \providecommand{\condname}{Condition}
  \providecommand{\probname}{Problem}

\makeatother

  \providecommand{\lemmaname}{Lemma}
  \providecommand{\corollaryname}{Corollary}
\providecommand{\theoremname}{Theorem}

\usepackage{endnotes}
\usepackage{color}
\usepackage{xparse}

\newtheorem{defn}{Definition}

\title{Scaling positive random matrices: concentration and asymptotic convergence}

\author{ Boris Landa\\
\small{Program in Applied Mathematics, Department of Mathematics, Yale University}\\
\small{boris.landa@yale.edu}
}

\begin{document}

\maketitle

\begin{abstract}
It is well known that any positive matrix can be scaled to have prescribed row and column sums by multiplying its rows and columns by certain positive scaling factors (which are unique up to a positive scalar). This procedure is known as matrix scaling, and has found numerous applications in operations research, economics, image processing, and machine learning. In this work, we investigate the behavior of the scaling factors and the resulting scaled matrix when the matrix to be scaled is random. Specifically, letting $\widetilde{A}\in\mathbb{R}^{M\times N}$ be a positive and bounded random matrix whose entries assume a certain type of independence, we provide a concentration inequality for the scaling factors of $\widetilde{A}$ around those of $A = \mathbb{E}[\widetilde{A}]$. This result is employed to bound the convergence rate of the scaling factors of $\widetilde{A}$ to those of $A$, as well as the concentration of the scaled version of $\widetilde{A}$ around the scaled version of $A$ in operator norm, as $M,N\rightarrow\infty$. 
When the entries of $\widetilde{A}$ are independent, $M=N$, and all prescribed row and column sums are $1$ (i.e., doubly-stochastic matrix scaling), both of the previously-mentioned bounds are $\mathcal{O}(\sqrt{\log N / N})$ with high probability.
We demonstrate our results in several simulations.
\end{abstract}

\section{Introduction}
Let $A\in \mathbb{R}^{M\times N}$ be a nonnegative matrix. It was established in a series of classical papers~\cite{sinkhorn1964relationship,sinkhorn1967diagonal,sinkhorn1967concerning,brualdi1966diagonal,brualdi1974dad,marshall1968scaling} that under certain conditions one can find a positive vector $\mathbf{x} = [x_1,\ldots,x_M]$ and a positive vector $\mathbf{y} = [y_1,\ldots,y_N]$, such that the matrix $P = D(\mathbf{x}) A D(\mathbf{y})$ has prescribed row sums $\mathbf{r} = [r_1,\ldots,r_M]$ and column sums $\mathbf{c} = [c_1,\ldots,c_N]$, where $D(\mathbf{v})$ is a diagonal matrix with $\mathbf{v}$ on its main diagonal. The problem of finding the appropriate $\mathbf{x}$ and $\mathbf{y}$ that produce $P$ with the prescribed row and column sums is known as \textit{matrix scaling} or \textit{matrix balancing}; see~\cite{idel2016review} for a comprehensive review of the topic and its various extensions. Formally, we use the following definition.
\begin{defn} [Matrix scaling] \label{def:matrix scaling}
We say that a pair of vectors $(\mathbf{x},\mathbf{y})$ \textit{scales} $A$ to row sums $\mathbf{r}$ and column sums $\mathbf{c}$, if
\begin{equation}
    r_i = \sum_{j=1}^N P_{i,j} = \sum_{j=1}^N x_i A_{i,j} y_j, \qquad \text{and} \qquad c_j = \sum_{i=1}^M P_{i,j} = \sum_{i=1}^M x_i A_{i,j} y_j, \label{eq:equations for x and y in matrix scaling definition}
\end{equation}
for all $i\in[M]$ and  $j\in [N]$. We refer to $\mathbf{x}$ and $\mathbf{y}$ from~\eqref{eq:equations for x and y in matrix scaling definition} (or their entries) as \textit{scaling factors} of $A$.
\end{defn}
In the special case that $M = N$ and $r_i=c_j=1$ for all $i\in[M]$ and $j\in [N]$, the problem of matrix scaling becomes that of finding a doubly-stochastic normalization of $A$, originally studied by Sinkhorn~\cite{sinkhorn1964relationship} with the motivation of estimating doubly-stochastic transition probability matrices. 

It is important to mention that~\eqref{eq:equations for x and y in matrix scaling definition} is a system of nonlinear equations in $\mathbf{x}$ and $\mathbf{y}$ with no closed-form solution. Nevertheless, if the scaling factors $\mathbf{x}$ and $\mathbf{y}$ exist, they can be found by the Sinkhorn-Knopp algorithm~\cite{sinkhorn1967concerning} (also known as the RAS algorithm), which is a simple iterative procedure that alternates between computing $\mathbf{x}$ via~\eqref{eq:equations for x and y in matrix scaling definition} using $\mathbf{y}$ from the previous iteration, and vice versa (a procedure equivalent to alternating between normalizing the rows of $A$ and normalizing the columns of $A$ to have the prescribed row and column sums, respectively).

From a theoretical perspective, given a nonnegative matrix $A$, existence and uniqueness of the scaling factors and of the scaled matrix $P$ depend primarily on the particular zero-pattern of $A$; see~\cite{brualdi1974dad} and references therein for more details.
In this work, we focus on the simpler case that $A$ is strictly positive, in which case existence and uniqueness of the scaling factors and of the scaled matrix $P$ are guaranteed by the following theorem (see~\cite{sinkhorn1967diagonal}). 
\begin{thm}[Existence and uniqueness~\cite{sinkhorn1967diagonal}] \label{thm:existence and uniquness of matrix scaling}
Suppose that $A$, $\mathbf{r}$, and $\mathbf{c}$ are positive, and $\Vert \mathbf{r} \Vert_1 = \Vert \mathbf{c} \Vert_1$. Then, there exists a pair of positive vectors $(\mathbf{x},\mathbf{y})$  that  {scales} $A$ to row sums $\mathbf{r}$ and column sums $\mathbf{c}$. Furthermore, the resulting scaled matrix $P = D(\mathbf{x}) A D(\mathbf{y})$ is unique, and the pair $(\mathbf{x},\mathbf{y})$ can be replaced only with $(\alpha \mathbf{x},\alpha^{-1} \mathbf{y}$), for any $\alpha > 0$.
\end{thm}
Over the years, matrix scaling and the Sinkhorn-Knopp algorithm have found a wide array of applications in science and engineering. In economy and operations research, classical applications of matrix scaling include transportation planning~\cite{lamond1981bregman}, analyzing migration fields~\cite{slater1984measuring}, and estimating social accounting matrices~\cite{schneider1990comparative}. In image processing and computer vision, matrix scaling was employed for image denoising ~\cite{milanfar2013symmetrizing} and graph matching~\cite{cour2006balanced}. Recently, matrix scaling has been attracting a growing interest from the machine learning community, with applications in manifold learning~\cite{marshall2019manifold,wormell2020spectral}, clustering~\cite{zass2006doubly,lim2020doubly}, and classification~\cite{frogner2015learning}. See also~\cite{peyre2019computational, cuturi2013sinkhorn} for applications of matrix scaling in data science through the machinery of optimal transport. 

In many practical situations, matrix scaling is actually applied to a random matrix that represents a perturbation, or a random observation, of an underlying deterministic population matrix; see for example~\cite{landa2020doubly,wormell2020spectral,milanfar2013symmetrizing,cho2010reweighted}. Arguably, this is the case in all of the previously-mentioned applications of matrix scaling whenever real data is involved. In particular, applications of matrix scaling in machine learning and data science often involve large data matrices that suffer from corruptions and measurement errors, and hence are more accurately described by random models. Due to such challenges, it is important to understand the influence of random perturbations in $A$ on the required scaling factors and on the resulting scaled matrix, particularly in the setting where $A$ is large and the entrywise perturbations are not necessarily small. It is noteworthy that existing literature related to scaling random matrices is mostly concerned with special cases such as the scaling of symmetric kernel matrices~\cite{wormell2020spectral,landa2020doubly} and the spectral properties of random  doubly-stochastic matrices~\cite{nguyen2014random,cappellini2009random}.

Let $\widetilde{A}\in\mathbb{R}^{M\times N}$ be a positive random matrix, and define $A = \mathbb{E}[\widetilde{A}]$. Theorem~\ref{thm:existence and uniquness of matrix scaling} establishes the existence and uniqueness of a set of scaling factors $\{(\alpha \mathbf{x},\alpha^{-1}\mathbf{y})\}_{\alpha>0}$ of $A$, together with the existence and uniqueness of the corresponding scaled matrix $P$. Theorem~\ref{thm:existence and uniquness of matrix scaling} can also be applied analogously to $\widetilde{A}$, establishing the existence and uniqueness of a set of random scaling factors $\{(\alpha \widetilde{\mathbf{x}},\alpha^{-1}\widetilde{\mathbf{y}})\}_{\alpha>0}$ of $\widetilde{A}$, as well as the existence and uniqueness of the corresponding scaled random matrix $\widetilde{P} = D(\widetilde{\mathbf{x}}) \widetilde{A} D(\widetilde{\mathbf{y}})$. The main purpose of this work is to establish that under suitable conditions on $\widetilde{A}$, $\mathbf{r}$, and $\mathbf{c}$, there is a pair of scaling factors $(\widetilde{\mathbf{x}},\widetilde{\mathbf{y}})$ of $\widetilde{A}$ that concentrates around a pair of scaling factors $({\mathbf{x}},{\mathbf{y}})$ of $A$ (in an appropriate sense), and furthermore, the resulting scaled random matrix $\widetilde{P}$ concentrates around $P$ in operator norm.
Notably, the main technical challenge in deriving such results is the implicit nonlinear representation of $\mathbf{x}$ and $\mathbf{y}$ in~\eqref{eq:equations for x and y in matrix scaling definition}, which prohibits the direct application of standard concentration inequalities. Therefore, an important aspect of this work is providing a mechanism for applying standard vector and matrix concentration inequalities in the analysis of random matrix scaling.

The main contributions of this work are as follows. We begin by providing a concentration inequality for the scaling factors of $\widetilde{A}$ around those of $A$ assuming the entries of $\widetilde{A}$ are bounded from above and from below away from zero, and in addition that they satisfy the property of being independent within each row and each column of $\widetilde{A}$ separately; see Theorem~\ref{thm:deviation of scaling factors} in Section~\ref{sec:concentration of scaling factors}. To that end, we derive a result concerning the stability of the scaling factors of a matrix under perturbations in the prescribed row and column sums, which may be of independent interest; see Lemma~\ref{lem:closeness of scaling factors} in Section~\ref{sec:supporting lemmas}. We then turn to consider an asymptotic setting of $M,N\rightarrow \infty$, and employ Theorem~\ref{thm:deviation of scaling factors} to bound the pointwise convergence rate of the scaling factors of $\widetilde{A}$ to those of $A$; see Theorem~\ref{thm:relative entry-wise convergence of scaling factors} and Corollary~\ref{cor:relative convergence of scaling factors} in Section~\ref{sec:convergence of scaling factors}. In addition, under the same asymptotic setting as above but further assuming that all entries of $\widetilde{A}$ are independent, we make use of Theorem~\ref{thm:relative entry-wise convergence of scaling factors} to bound the asymptotic concentration of $\widetilde{P}$ around $P$ in operator norm; see Theorem~\ref{thm:convergence of scaled matrix in operator norm} and Corollary~\ref{cor:relative convergence of scaled matrix in operator norm} in Section~\ref{sec:convergence of scaled matrix}. 
We conclude by conducting several numerical experiments that corroborate our theoretical findings and demonstrate that our convergence rates are tight in certain situations; see Section~\ref{sec:numerical experiments}.

\section{Main results}

\subsection{Concentration of matrix scaling factors} \label{sec:concentration of scaling factors}

Let us define 
\begin{equation}
    \overline{r}_i = \frac{r_i}{\sqrt{\Vert \mathbf{r} \Vert_1}}, \qquad\qquad \overline{c}_j = \frac{c_j}{\sqrt{\Vert \mathbf{c} \Vert_1}}, \label{eq:r_overline and c_overline def}
\end{equation}
for all  $i\in [M]$ and  $j\in [N]$.
The following lemma describes a useful normalization of the scaling factors of $A$ and the resulting bounds on their entries.
\begin{lem}[Boundedness of scaling factors] \label{lem:boundedness of scaling factors}
Suppose that $A$, $\mathbf{r}$, and $\mathbf{c}$ are positive, and $\Vert \mathbf{r} \Vert_1 = \Vert \mathbf{c} \Vert_1$. Then, there exists a unique pair of positive vectors $(\mathbf{x},\mathbf{y})$ that satisfies $\Vert \mathbf{x} \Vert_1 = \Vert \mathbf{y} \Vert_1$ and scales $A$ to row sums $\mathbf{r}$ and column sums $\mathbf{c}$. Furthermore, denoting $a = \min_{i,j} A_{i,j}$ and $b = \max_{i,j} A_{i,j}$, we have that for all $i\in [M]$ and $j\in [N]$:
    \begin{equation}
        \frac{\sqrt{a}}{b} \leq \frac{x_i}{\overline{r}_i} \leq \frac{\sqrt{b}}{a}, \qquad \qquad  \frac{\sqrt{a}}{b} \leq \frac{y_j}{\overline{c}_j} \leq \frac{\sqrt{b}}{a}. \label{eq:x_i and y_i bounds}
    \end{equation}
\end{lem}
The proof can be found in Section~\ref{sec:proof of boundednes lemma} and is based on a straightforward manipulation of the system of equations in~\eqref{eq:equations for x and y in matrix scaling definition}. The normalization $\Vert \mathbf{x} \Vert_1 = \Vert \mathbf{y} \Vert_1$ is natural and convenient, since it provides a symmetric bound for the entries of $\mathbf{x}$ and $\mathbf{y}$ in terms of $a$ and $b$ while precisely characterizing their magnitudes according to $\mathbf{r}$ and $\mathbf{c}$, respectively. Note that the condition $\Vert \mathbf{r} \Vert_1 = \Vert \mathbf{c} \Vert_1$ in Theorem~\ref{thm:existence and uniquness of matrix scaling} and in Lemma~\ref{lem:boundedness of scaling factors} is necessary for the existence of the scaling factors, as by Definition~\ref{def:matrix scaling} each of the quantities $\Vert \mathbf{r} \Vert_1$ and $\Vert \mathbf{c} \Vert_1$ should be the sum of all entries in the scaled matrix $P = D(\mathbf{x}) A D(\mathbf{y})$. From this point onward we will always assume that $\mathbf{r}$ and $\mathbf{c}$ are positive and $\Vert \mathbf{r} \Vert_1 = \Vert \mathbf{c} \Vert_1$, denoting the sum of all entries in $P$ by
\begin{equation}
    s = \Vert \mathbf{r}\Vert_1 = \Vert \mathbf{c} \Vert_1. \label{eq: s def}
\end{equation}
We now have the following theorem, which provides a concentration inequality for a certain pair of scaling factors of $\widetilde{A}$ around the pair $(\mathbf{x},\mathbf{y})$ from Lemma~\ref{lem:boundedness of scaling factors} (taking $A = \mathbb{E}[\widetilde{A}]$).
\begin{thm} [Concentration of scaling factors]\label{thm:deviation of scaling factors}
Let $\widetilde{A}\in\mathbb{R}^{M\times N}$ be a positive random matrix, $A = \mathbb{E}[\widetilde{A}]$, and $(\mathbf{x},\mathbf{y})$ be the unique pair of positive vectors that satisfies $\Vert \mathbf{x} \Vert_1 = \Vert \mathbf{y} \Vert_1$ and scales $A$ to row sums $\mathbf{r}$ and column sums $\mathbf{c}$. Suppose that $\widetilde{A}_{i,j} \in [a_{i,j},b_{i,j}]$ a.s. for all $i\in [M]$ and $j\in [N]$, and denote $a = \min_{i,j} a_{i,j}$, $b = \max_{i,j} b_{i,j}$, and $d = \max_{i,j} \{b_{i,j} - a_{i,j}\}$. Suppose further that $\{\widetilde{A}_{i,j}\}_{j=1}^N$ are independent for each $i\in[M]$, and $\{\widetilde{A}_{i,j}\}_{i=1}^M$ are independent for each $j\in[N]$. Then, there exists a pair of positive random vectors $(\widetilde{\mathbf{x}},\widetilde{\mathbf{y}})$ that scales $\widetilde{A}$ to row sums $\mathbf{r}$ and column $\mathbf{c}$, such that for any $\delta \in (0, 1]$, with probability at least
\begin{equation}
    1 - 2 M \operatorname{exp}\left( -\frac{\delta^2 s^2}{ C_p^2 \Vert \mathbf{c} \Vert_2^2} \right) - 2 N \operatorname{exp}\left(-\frac{\delta^2 s^2}{ C_p^2 \Vert \mathbf{r} \Vert_2^2}\right), \label{eq:probability of tail bounds on x_tilde and y_tilde}
\end{equation}
we have that for all $i\in [M]$ and $j\in [N]$:
    \begin{equation}
        \frac{\vert \widetilde{x}_i - x_i \vert}{x_i} \leq  
      \frac{ C_e \delta s }{M \min_i r_i} , \qquad \qquad
     \frac{\vert \widetilde{y}_j - y_j \vert}{y_j} \leq   \frac{ C_e \delta s }{N \min_j {c}_j }, \label{eq:x_tilde_i and y_tilde_j relative error bounds}
    \end{equation}
where
\begin{equation}
    C_p = \sqrt{2}\left(\frac{b d}{a^2}\right), \qquad \qquad C_e = 1 + 2\left(\frac{b}{a} \right)^{7/2}. \label{eq:C_p and C_e expressions}
\end{equation}

\end{thm}
Note that Theorem~\ref{thm:deviation of scaling factors} requires that the entries of $\widetilde{A}$ are independent in each of its rows and each of its columns separately. This condition is clearly less restrictive than the requirement that all of the entries of $\widetilde{A}$ are independent. For instance, consider the matrix $\widetilde{A}_{i,j} = g_{i,j}(u_i v_j)$, where $\{u_i\}_{i=1}^M$, $\{v_j\}_{j=1}^N$ are independent Rademacher variables, and $g_{i,j}:\{-1,1\} \rightarrow \{a_{i,j},b_{i,j}\}$ are deterministic functions with $0 < a_{i,j} < b_{i,j}$, for all $i\in [M]$, $j\in [N]$. Evidently, each row and column of $\widetilde{A}$ contains independent entries, yet the entries of $\widetilde{A}$ are strongly dependant since knowing any single row (column) of $\widetilde{A}$ substantially restricts the distribution of any other row (column).

It is worthwhile to point out that Theorem~\ref{thm:deviation of scaling factors} also implies the following statement, which is perhaps more intuitive than the formulation in Theorem~\ref{thm:deviation of scaling factors}. For any pair of scaling factors $(\widetilde{\mathbf{x}}^{'},\widetilde{\mathbf{y}}^{'})$ of $\widetilde{A}$, Theorem~\ref{thm:deviation of scaling factors} implies that there exists a pair of scaling factors $({\mathbf{x}}^{'},{\mathbf{y}}^{'})$ of $A$ such that with probability at least~\eqref{eq:probability of tail bounds on x_tilde and y_tilde} the bounds in~\eqref{eq:x_tilde_i and y_tilde_j relative error bounds} hold if we replace $(\widetilde{\mathbf{x}},\widetilde{\mathbf{y}})$ and $({\mathbf{x}},{\mathbf{y}})$ with $(\widetilde{\mathbf{x}}^{'},\widetilde{\mathbf{y}}^{'})$ and $({\mathbf{x}}^{'},{\mathbf{y}}^{'})$, respectively (under to the conditions in Theorem~\ref{thm:deviation of scaling factors}). This claim stems simply from the fact that any pair $(\widetilde{\mathbf{x}}^{'},\widetilde{\mathbf{y}}^{'})$ of scaling factors of $\widetilde{A}$ can be written as $(\alpha \widetilde{\mathbf{x}},\alpha^{-1} \widetilde{\mathbf{y}})$ for some $\alpha>0$, where $( \widetilde{\mathbf{x}},\widetilde{\mathbf{y}})$ is the specific pair whose existence is guaranteed by Theorem~\ref{thm:deviation of scaling factors}. Subsequently, taking $(\mathbf{x}^{'},\mathbf{y}^{'}) = (\alpha\mathbf{x},\alpha^{-1}\mathbf{y})$, where $( {\mathbf{x}},{\mathbf{y}})$ is as in Theorem~\ref{thm:deviation of scaling factors}, gives that $|\widetilde{x}^{'}_i - x_i^{'}|/x_i^{'} = |\widetilde{x}_i - x_i|/x_i$ and $|\widetilde{y}^{'}_j - y_j^{'}|/y_j^{'} = |\widetilde{y}_j - y_j|/y_j$. 

The proof of Theorem~\ref{thm:deviation of scaling factors} can be found in Section~\ref{sec:proof of main theorem} and is based on the following idea, which is a simple two-step procedure. First, we use Lemma~\ref{lem:boundedness of scaling factors} in conjunction with Hoeffding's inequality~\cite{hoeffding1994probability} to prove that the row and column sums of $D(\mathbf{x}) \widetilde{A} D(\mathbf{y})$ concentrate around $\mathbf{r}$ and $\mathbf{c}$, respectively; see Lemma~\ref{lem:deviation from scaling} in Section~\ref{sec:proof of main theorem}. Second, we prove that if the matrix $\widetilde{A}$ can be approximately scaled by the pair $(\mathbf{x},\mathbf{y})$, it must imply that $(\mathbf{x},\mathbf{y})$ is sufficiently close to a pair of scaling factors of $\widetilde{A}$. This result is based on Sinkhorn's technique in~\cite{sinkhorn1967diagonal} for proving the uniqueness of the scaling factors, which we substantially extend to describe the stability of the scaling factors under approximate scaling (or equivalently, under perturbations of the prescribed row and column sums); see Lemma~\ref{lem:closeness of scaling factors} in Section~\ref{sec:proof of main theorem}. It is worthwhile to point out that Hoeffding's inequality in the proof of Theorem~\ref{thm:deviation of scaling factors} can be replaced with any other concentration inequality for sums of random variables, allowing one to relax the assumptions on boundedness and independence. 

\subsection{Asymptotic convergence of scaling factors } \label{sec:convergence of scaling factors}
We now place ourselves in an asymptotic setting where the dimensions of $\widetilde{A}$ tend to infinity, and apply Theorem~\ref{thm:deviation of scaling factors} to study the asymptotic convergence of the scaling factors of $\widetilde{A}$ to those of $A$ in relative error. Let $\{ \widetilde{A}^{(N)} \in \mathbb{R}^{M_N \times N} \}_{N=N_0}^\infty$ be a sequence of positive random matrices and define $A^{(N)} = \mathbb{E}[\widetilde{A}^{(N)}]$, where $N_0$ is some positive integer, and $\lim_{N\rightarrow \infty} M_N = \infty$. Suppose that for any positive integer $N \geq N_0$, we are given positive row sums $\mathbf{r}^{(N)}$ and column sums $\mathbf{c}^{(N)}$ that satisfy $\Vert \mathbf{r}^{(N)} \Vert_1 = \Vert \mathbf{c}^{(N)} \Vert_1$. According to Lemma~\ref{lem:boundedness of scaling factors}, $A^{(N)}$ can be scaled to row sums $\mathbf{r}^{(N)}$ and column sums $\mathbf{c}^{(N)}$ by a unique pair of positive vectors $(\mathbf{x}^{(N)},\mathbf{y}^{(N)})$ that satisfies $\Vert \mathbf{x}^{(N)}\Vert_1 = \Vert \mathbf{y}^{(N)} \Vert_1$. 
As a measure of discrepancy between $(\mathbf{x}^{(N)},\mathbf{y}^{(N)})$ and another pair of vectors $(\widetilde{\mathbf{x}},\widetilde{\mathbf{y}})$, where $\widetilde{\mathbf{x}}\in \mathbb{R}^{M_N}$ and $\widetilde{\mathbf{y}}\in \mathbb{R}^{N}$, we define
\begin{equation}
    \mathcal{E}_N(\widetilde{\mathbf{x}},\widetilde{\mathbf{y}}) = \max \left\{ \max_{i\in [M_N]} \frac{\vert \widetilde{x}_i - x_i^{(N)} \vert}{x_i^{(N)}}, \max_{j\in [N]} \frac{\vert \widetilde{y}_j - y_j^{(N)} \vert}{y_j^{(N)}} \right\}. \label{eq:E_mathcal def}
\end{equation}
It is important to mention that the scaling factors $x_i^{(N)}$ and $y_j^{(N)}$ of $A^{(N)}$ can potentially converge to $0$ or grow unbounded as $N\rightarrow\infty$, depending on the asymptotic behavior of the prescribed row sums $\mathbf{r}^{(N)}$ and column sums $\mathbf{c}^{(N)}$. Consequently, the normalizations by $x_i^{(N)}$ and $y_j^{(N)}$ appearing in the error measure~\eqref{eq:E_mathcal def} are important for making $\mathcal{E}_N(\widetilde{\mathbf{x}},\widetilde{\mathbf{y}})$ meaningful in the asymptotic regime of $N\rightarrow\infty$.
Let us denote $  s^{(N)} = \Vert \mathbf{r}^{(N)} \Vert_1 = \Vert \mathbf{c}^{(N)} \Vert_1$, and define the quantities
\begin{equation}
    \rho_1^{(N)} = \max \left\{\frac{\Vert \mathbf{r}^{(N)} \Vert_2}{s^{(N)}}, \frac{\Vert \mathbf{c}^{(N)} \Vert_2}{s^{(N)}} \right\} , \qquad \qquad \rho_2^{(N)} = \max \left\{\frac{ s^{(N)} }{M_N \min_i {r}_i^{(N)}}, \frac{ s^{(N)} }{N \min_j {c}_j^{(N)}} \right\}. \label{eq:rho_1 and rho_2 def}
\end{equation}
In what follows we use the notation $X_N = \mathcal{O}_{\text{w.h.p}}(\gamma_N)$, where $\{X_N\}$ is a sequence of random variables and $\{\gamma_N\}$ is a deterministic sequence, to mean \textit{order with high probability}, namely that there exists a constant $C$ such that $\lim_{N\rightarrow\infty}\operatorname{Pr}\{X_N \leq C \gamma_N\} = 1$. Note that $\mathcal{O}_{\text{w.h.p}}(\cdot)$ is not equivalent to \textit{order in probability} $\mathcal{O}_p(\cdot)$~\cite{mann1943stochastic}. In particular, $X_N = \mathcal{O}_{\text{w.h.p}}(\gamma_N)$ implies that $X_N = \mathcal{O}_{p}(\gamma_N)$ but not the other way around.

We now have the following theorem, which provides a bound on the convergence rate of a certain sequence of scaling factors $\{(\widetilde{\mathbf{x}}^{(N)},\widetilde{\mathbf{y}}^{(N)})\}$ of $\{\widetilde{A}^{(N)}\}$ to $\{(\mathbf{x}^{(N)},\mathbf{y}^{(N)})\}$.
\begin{thm} [Convergence rate of scaling factors] \label{thm:relative entry-wise convergence of scaling factors}
Suppose that for all indices $N$ the matrices $\{\widetilde{A}^{(N)}\}$ satisfy the conditions in Theorem~\ref{thm:deviation of scaling factors} with universal positive constants $a,b$ (independent of $N$).
Then, there exists a sequence of scaling factors $\{(\widetilde{\mathbf{x}}^{(N)},\widetilde{\mathbf{y}}^{(N)})\}$ of $\{\widetilde{A}^{(N)}\}$, such that
\begin{equation}
    \mathcal{E}_N(\widetilde{\mathbf{x}}^{(N)},\widetilde{\mathbf{y}}^{(N)}) = \mathcal{O}_{\text{w.h.p}}\left( \rho_1^{(N)} \rho_2^{(N)}\sqrt{{\log (\max \{ M_N, N \}) } }\right). \label{eq:E_mathcal convergence rate}
\end{equation}
\end{thm}
The proof can be found in Section~\ref{sec:proof of theorem on convergence rate of scaling factors} and is largely based on a direct application of Theorem~\ref{thm:deviation of scaling factors} using an appropriate $\delta$.

To exemplify Theorem~\ref{thm:relative entry-wise convergence of scaling factors}, let us consider the setting of {doubly-stochastic} matrix scaling, namely $M_N=N$, and $r_i^{(N)}=c_j^{(N)}=1$ for all $i\in[M]$, $j\in [N]$. According to~\eqref{eq:rho_1 and rho_2 def} we have $\rho_1^{(N)} = 1/\sqrt{N}$ and $\rho_2^{(N)} = 1$. Hence, Theorem~\ref{thm:relative entry-wise convergence of scaling factors} asserts that there exists a sequence of scaling factors $\{(\widetilde{\mathbf{x}}^{(N)},\widetilde{\mathbf{y}}^{(N)})\}$ of $\{\widetilde{A}^{(N)}\}$, such that $\mathcal{E}_N(\widetilde{\mathbf{x}}^{(N)},\widetilde{\mathbf{y}}^{(N)}) = \mathcal{O}_{\text{w.h.p}} (\sqrt{{\log N}/{N}})$.
Similarly, it is easy to verify that the same convergence rate of $\mathcal{O}_{\text{w.h.p}} \left(\sqrt{{\log N}/{N}}\right)$ holds whenever $M_N$ grows proportionally to $N$ and $\max_i r_i^{(N)} / \min_i r_i^{(N)} \leq C$, $\max_j c_j^{(N)}/\min_j c_j^{(N)} \leq C$, for some universal constant $C$. If instead $M_N$ grows disproportionately to $N$, the convergence rate of  $(\widetilde{\mathbf{x}}^{(N)},\widetilde{\mathbf{y}}^{(N)})$ to $({\mathbf{x}}^{(N)},{\mathbf{y}}^{(N)})$ is dominated by the minimum between $M_N$ and $N$, as described in the next corollary of Theorem~\ref{thm:relative entry-wise convergence of scaling factors}.
\begin{cor} \label{cor:relative convergence of scaling factors}
Suppose that the conditions in Theorem~\ref{thm:relative entry-wise convergence of scaling factors} hold, and in addition $\max_i r_i^{(N)} / \min_i r_i^{(N)} \leq C$, $\max_j c_j^{(N)} /\min_j c_{j}^{(N)} \leq C$, for all indices $N$ and some universal constant $C$ (independent of $N$). Then,
\begin{equation}
    \mathcal{E}_N(\widetilde{\mathbf{x}}^{(N)},\widetilde{\mathbf{y}}^{(N)}) = \mathcal{O}_{\text{w.h.p}}\left(\sqrt{\frac{\log (\max \{ M_N, N \}) }{\min \{ M_N, N \}}}\right). \label{eq:scaling factors convergence rate for equal row and column sums}
\end{equation}
\end{cor} 
The proof follows immediately from the fact that $\rho_1^{(N)} \leq C \max \{ {M_N}^{-1/2} , {N}^{-1/2}\}$ and $\rho_2^{(N)} \leq C$ if $\max_i r_i^{(N)} / \min_i r_i^{(N)} \leq C$ and $\max_j c_j^{(N)}/\min_j c_j^{(N)} \leq C$.

Aside from the setting where the $r_i^{(N)}$'s and $c_j^{(N)}$'s have the same orders of magnitude, Theorem~\ref{thm:relative entry-wise convergence of scaling factors} provides guarantees on the convergence rate of the scaling factors even if some of the $r_i^{(N)}$'s or $c_j^{(N)}$'s grow unbounded with $N$ relative to others. For instance, considering again the setting of doubly-stochastic matrix scaling, we can set a fixed number of the $r_i^{(N)}$'s or $c_j^{(N)}$'s to be $\sqrt{N}$ instead of $1$ (for all indices $N$), without affecting the behavior of $\Vert \mathbf{r}^{(N)} \Vert_2$, $\Vert \mathbf{c}^{(N)} \Vert_2$, and $s^{(N)}$ asymptotically as $N$ grows. Consequently, the convergence rate of $(\widetilde{\mathbf{x}}^{(N)},\widetilde{\mathbf{y}}^{(N)})$ to $({\mathbf{x}}^{(N)},{\mathbf{y}}^{(N)})$ in this case would remain $\mathcal{O}_{\text{w.h.p}}(\sqrt{\log N / N})$.

\subsection{Concentration of $\widetilde{P}$ around $P$ in operator norm} \label{sec:convergence of scaled matrix}
Let $\widetilde{P}^{(N)}$ and ${P}^{(N)}$ be the matrices obtained from $\widetilde{A}^{(N)}$ and $A^{(N)}$, respectively, after scaling them to row sums $\mathbf{r}^{(N)}$ and column sums $\mathbf{c}^{(N)}$, i.e.,
\begin{equation}
    \widetilde{P}^{(N)} = D(\widetilde{\mathbf{x}}^{(N)} )\widetilde{A}^{(N)} D(\widetilde{\mathbf{y}}^{(N)} ), \qquad\qquad P^{(N)} = D({\mathbf{x}}^{(N)} ) A^{(N) } D({\mathbf{y}}^{(N)} ),
\end{equation}
where $(\widetilde{\mathbf{x}}^{(N)},\widetilde{\mathbf{y}}^{(N)})$ is any pair of scaling factors of $\widetilde{A}$. Note that by Theorem~\ref{thm:existence and uniquness of matrix scaling} the matrices $\widetilde{P}^{(N)}$ and ${P}^{(N)}$ are uniquely determined ($\widetilde{P}^{(N)}$ being a random matrix). In addition, we define the quantity
\begin{equation}
      \rho_3^{(N)} = \frac{\sqrt{M_N}\max_i r_i^{(N)} \cdot \sqrt{N} \max_j c_j^{(N)}}{s^{(N)}}.  \label{eq:rho_3 def}
  \end{equation}
We now have the following result, which provides an upper bound on the concentration of $\widetilde{P}^{(N)}$ around ${P}^{(N)}$ in operator norm.
\begin{thm} (Asymptotic concentration of $\widetilde{P}^{(N)}$ around $P^{(N)}$) \label{thm:convergence of scaled matrix in operator norm}
Suppose that for each index $N$ the entries of $\widetilde{A}^{(N)}$ are independent, and $\widetilde{A}_{i,j}^{(N)} \in [a,b]$ a.s. for all $i\in[M_N]$, $j\in[N]$, and some universal positive constants $a,b$ (independent of $N$). 
Then,
\begin{equation}
    \Vert \widetilde{P}^{(N)} - P^{(N)} \Vert_2 
    = \mathcal{O}_{\text{w.h.p}} \left( \rho_1^{(N)} \rho_2^{(N)} \rho_3^{(N)} \sqrt{{\log (\max \{ M_N, N \}) } } \right).
    \label{eq:scaled matrix convergence rate in operator norm}
\end{equation}
\end{thm}
The proof can be found in Section~\ref{sec:proof of theorem on convergence in operator norm}, and is based on Theorem~\ref{thm:relative entry-wise convergence of scaling factors} and the concentration of $\widetilde{A}^{(N)}$ around $A^{(N)}$ in operator norm.
To exemplify Theorem~\ref{thm:convergence of scaled matrix in operator norm}, we consider again the case of doubly-stochastic matrix scaling, where we have $\rho_1^{(N)} = 1/\sqrt{N}$, and $\rho_2^{(N)} = \rho_3^{(N)} = 1$. Therefore, $ \Vert \widetilde{P}^{(N)} - P^{(N)} \Vert_2 = \mathcal{O}_{\text{w.h.p}} ( \sqrt{{\log N}/{{N}}} )$.
Note that since $P^{(N)}$ is doubly-stochastic, it follows that $\Vert P^{(N)} \Vert_2 = 1$ (see~\cite{horn2012matrix}).
In the more general case where the prescribed row and column sums are not $1$, the operator norm of $P^{(N)}$ can converge to zero or grow unbounded with $N$, depending on the asymptotic behavior of the prescribed row and column sums. Consequently, we also consider the normalized error $\Vert \widetilde{P}^{(N)} - P^{(N)} \Vert_2/\Vert P^{(N)} \Vert_2$, which is the subject of the following Corollary of Theorem~\ref{thm:convergence of scaled matrix in operator norm} for the case that all prescribed row sums and all prescribed column sums admit the same behaviors with $N$.
\begin{cor} \label{cor:relative convergence of scaled matrix in operator norm}
Suppose that the conditions in Theorem~\ref{thm:convergence of scaled matrix in operator norm} hold, and in addition $\max_i r_i^{(N)} / \min_i r_i^{(N)} \leq C$, $\max_j c_j^{(N)} /\min_j c_{j}^{(N)} \leq C$, for all indices $N$ and some universal constant $C$ (independent of $N$). Then,
\begin{equation}
    \frac{\Vert \widetilde{P}^{(N)} - P^{(N)} \Vert_2 }{\Vert P^{(N)} \Vert_2} = \mathcal{O}_{\text{w.h.p}}\left(\sqrt{\frac{\log (\max \{ M_N, N \}) }{\min \{ M_N, N \}}}\right). \label{eq:scaled matrix convergence rate in operator norm for equal row and column sums}
\end{equation}
\end{cor}
\begin{proof}
Observe that 
\begin{align}
    \Vert P^{(N)} \Vert_2 &\geq \max\left\{\left\Vert \frac{P^{(N)} \mathbf{1}_N}{\sqrt{N}}\right\Vert_2, \left\Vert \frac{\mathbf{1}^T_{M_N} P^{(N)}}{\sqrt{M_N}}\right\Vert_2\right\} \geq \max\left\{ \frac{\Vert \mathbf{r}^{(N)}\Vert_2}{\sqrt{N}}, \frac{ \Vert\mathbf{c}^{(N)}\Vert_2}{\sqrt{M_N}}\right\} \nonumber \\
    &\geq \sqrt{\min_i r_i^{(N)} \min_j c_j^{(N)}},
\end{align}
where $\mathbf{1}_N$ is the column vector of ones in $\mathbb{R}^{N}$. In addition, it is easy to verify that $\rho_1^{(N)} \leq C \max \{ {M_N}^{-1/2} , {N}^{-1/2}\}$, $\rho_2^{(N)} \leq C$, and 
\begin{equation}
    \frac{\rho_3^{(N)}}{\Vert P^{(N)} \Vert_2} \leq \frac{\max_i r_i^{(N)} \max_j c_j^{(N)}}{\Vert P^{(N)} \Vert_2 \sqrt{\min_i r_i^{(N)} \min_j c_j^{(N)}}} \leq C^2,
\end{equation}
where we used the fact that $s^{(N)} = \sqrt{s^{(N)}} \sqrt{s^{(N)}} \geq \sqrt{M_N \min_i r_i^{(N)}} \sqrt{N \min_j c_j^{(N)}}$. Applying Theorem~\ref{thm:convergence of scaled matrix in operator norm} and using all of the above concludes the proof.
\end{proof}.

\section{Numerical examples} \label{sec:numerical experiments}
We now exemplify our results in several simulations. In all of our experiments, the matrix $A^{(N)}$ was generated by sampling its entries independently and uniformly from $[1.5,2.5]$, and $\widetilde{A}_{i,j}^{(N)}$ were sampled independently and uniformly from $[A_{i,j}^{(N)} - 0.5, A_{i,j}^{(N)}+0.5]$ for all $i\in[M_N]$ and $j\in [N]$. Then, the Sinkhorn-Knopp algorithm~\cite{sinkhorn1967concerning,knight2008sinkhorn} was applied to both $A^{(N)}$ and $\widetilde{A}^{(N)}$, where the algorithm' iterations were terminated once the row and column sums of the scaled matrices reached their targets up to an error of $10^{-12}$. The resulting pairs of scaling factors were normalized so that $\Vert \widetilde{\mathbf{x}}^{(N)} \Vert_1 = \Vert \widetilde{\mathbf{y}}^{(N)} \Vert_1$ and $\Vert {\mathbf{x}}^{(N)} \Vert_1 = \Vert {\mathbf{y}}^{(N)} \Vert_1$. This process was repeated $20$ times (each time for a different realization of $A^{(N)}$ and $\widetilde{A}^{(N)}$) and the error measures appearing in the left-hand sides of~\eqref{eq:E_mathcal convergence rate} and~\eqref{eq:scaled matrix convergence rate in operator norm for equal row and column sums} were computed and averaged over the $20$ randomized trials.

Figure~\ref{fig:convergence of scaling factors} depicts the behavior of the empirical error $\mathcal{E}_N(\widetilde{\mathbf{x}}^{(N)},\widetilde{\mathbf{y}}^{(N)})$ (see~\eqref{eq:E_mathcal def}) as a function of $N$ in several scenarios. Specifically, Figure~\ref{fig:convergence of scaling factors doubly stochastic} exemplifies the scenario of doubly-stochastic matrix scaling, i.e., $M_N = N$ and $r_i^{(N)} = c_j^{(N)} = 1$ for all $i\in[M]$ and $j\in[N]$, in which case Corollary~\ref{cor:relative convergence of scaling factors} guarantees that $\mathcal{E}_N(\widetilde{\mathbf{x}}^{(N)},\widetilde{\mathbf{y}}^{(N)}) = \mathcal{O}_{\text{w.h.p}} (N^{-1/2}\sqrt{\log N})$. Figure~\ref{fig:convergence of scaling factors rectangular matrix and random sums} illustrates the case of a rectangular matrix with $M_N = 3N$, where the prescribed row and column sums were sampled independently and uniformly from $[0.1,1]$ and normalized to sum to $1$. In this case, since $M_N$ is proportional to $N$, and in addition $\max_i r_i^{(N)} / \min_i r_i^{(N)} \leq 10$, $\max_j c_j^{(N)} / \min_j c_j^{(N)} \leq 10$, Corollary~\ref{cor:relative convergence of scaling factors} again dictates that $\mathcal{E}_N(\widetilde{\mathbf{x}}^{(N)},\widetilde{\mathbf{y}}^{(N)}) = \mathcal{O}_{\text{w.h.p}} (N^{-1/2} \sqrt{\log N})$ as for the doubly-stochastic case. Figure~\ref{fig:convergence of scaling factors rectangular matrix} illustrates the scenario of a rectangular matrix with $M_N = 10\sqrt{N}$ and $r_i^{(N)} = N$, $c_j^{(N)} = M_N$, for all $i\in [M]$ and $j\in[N]$. In this case it follows from Corollary~\ref{cor:relative convergence of scaling factors} that $\mathcal{E}_N(\widetilde{\mathbf{x}}^{(N)},\widetilde{\mathbf{y}}^{(N)}) = \mathcal{O}_{\text{w.h.p}} (N^{-1/4} \sqrt{\log N})$. It is evident from Figures~\ref{fig:convergence of scaling factors doubly stochastic}, \ref{fig:convergence of scaling factors rectangular matrix and random sums}, \ref{fig:convergence of scaling factors rectangular matrix} that the asymptotic bound in Theorem~\ref{thm:relative entry-wise convergence of scaling factors} agrees very well with the experimental results, suggesting that this bound is tight for the considered scenarios, and in particular that the factor $\sqrt{\log(N)}$ in the corresponding bounds is necessary.   

\begin{figure} 
  \centering
  	{
  	\subfloat[$M_N = N$, $r_i^{(N)} = c_j^{(N)} = 1$]  
  	{
    \includegraphics[width=0.3\textwidth]{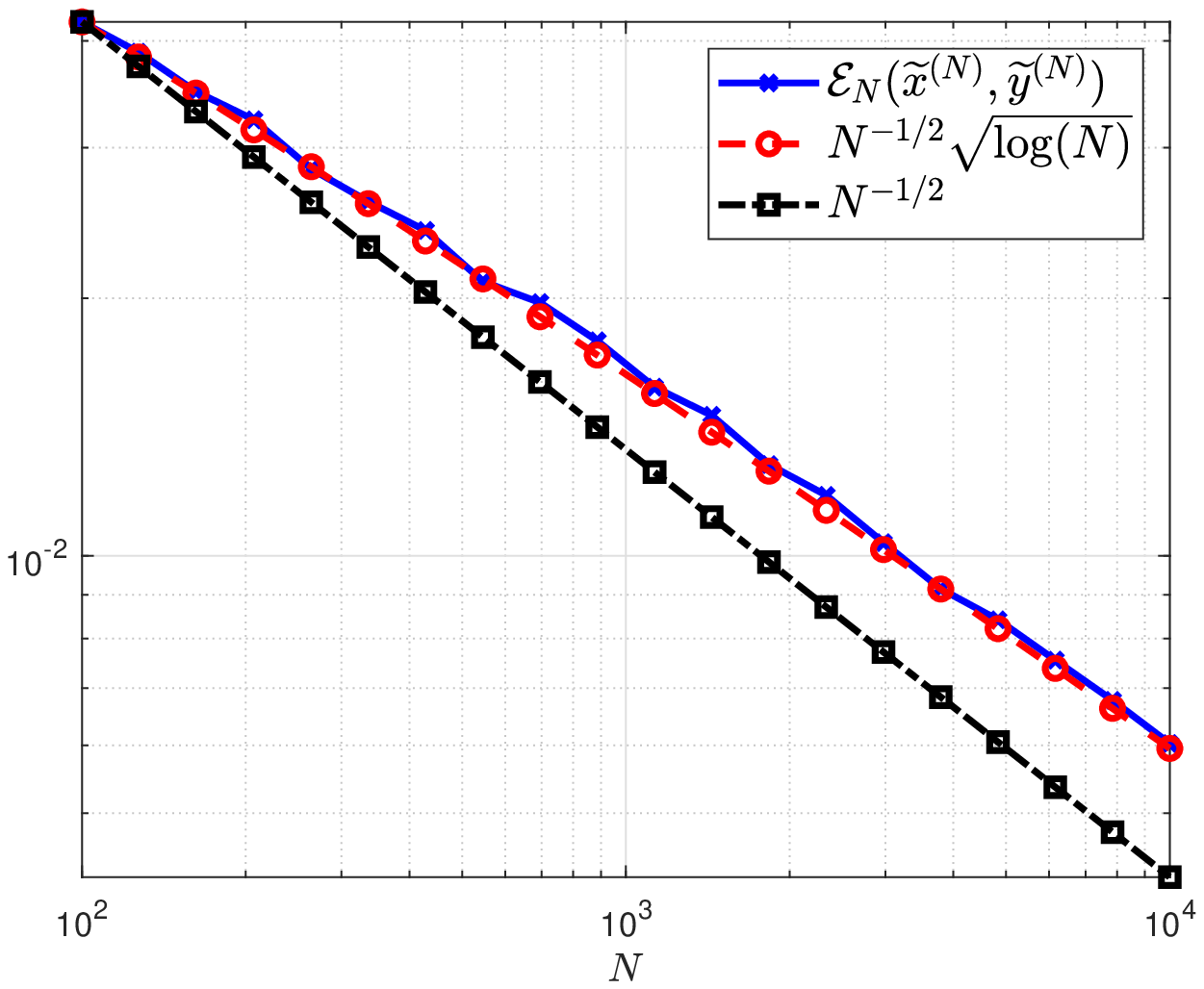} \label{fig:convergence of scaling factors doubly stochastic}
    }
    \subfloat[$M_N = 3N$, $r_i,c_j$ are random]  
  	{
    \includegraphics[width=0.3\textwidth]{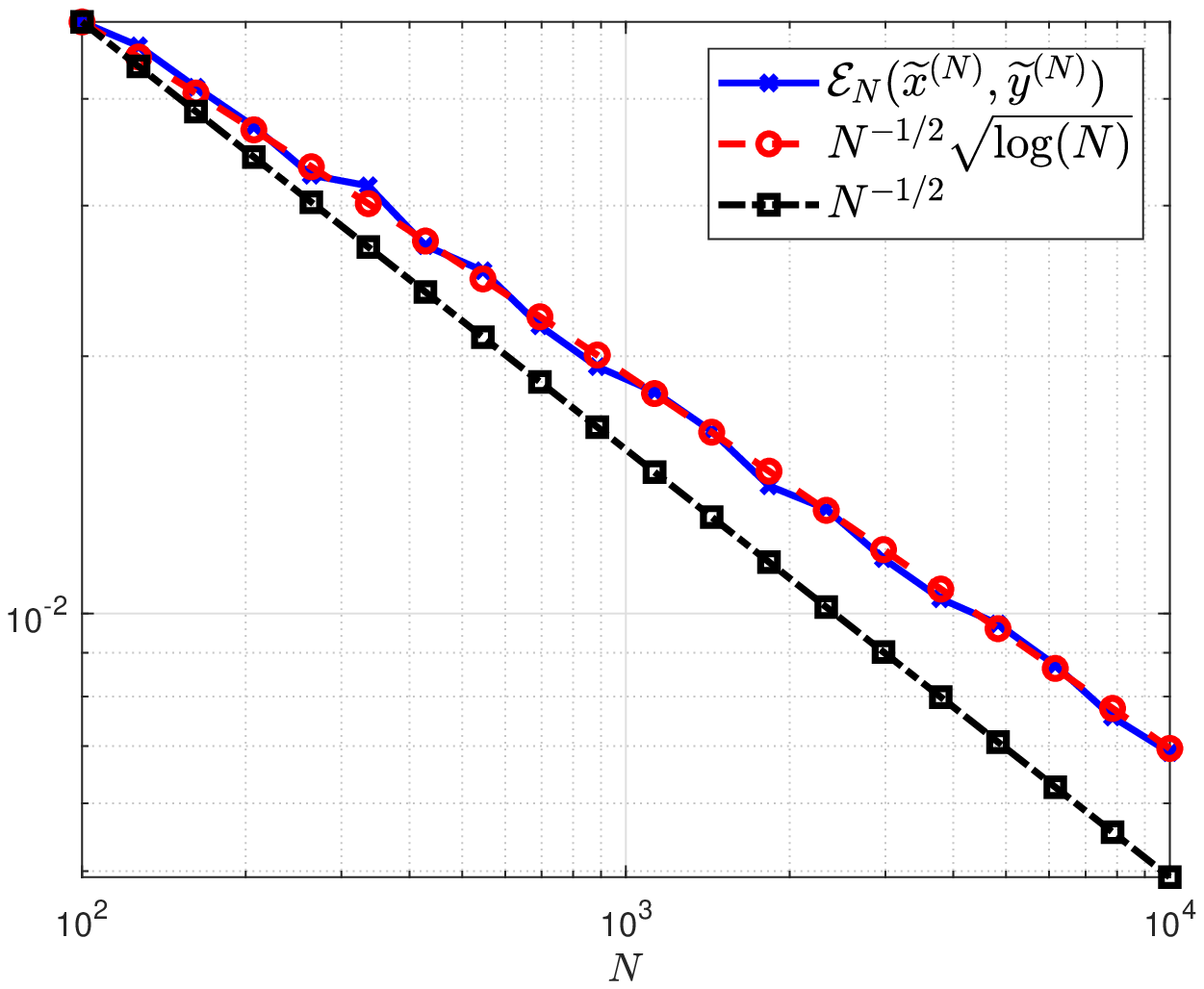} \label{fig:convergence of scaling factors rectangular matrix and random sums}
    }
    \subfloat[$M_N=10\sqrt{N}$, $r_i=N$, $c_j=M_N$]  
  	{
    \includegraphics[width=0.3\textwidth]{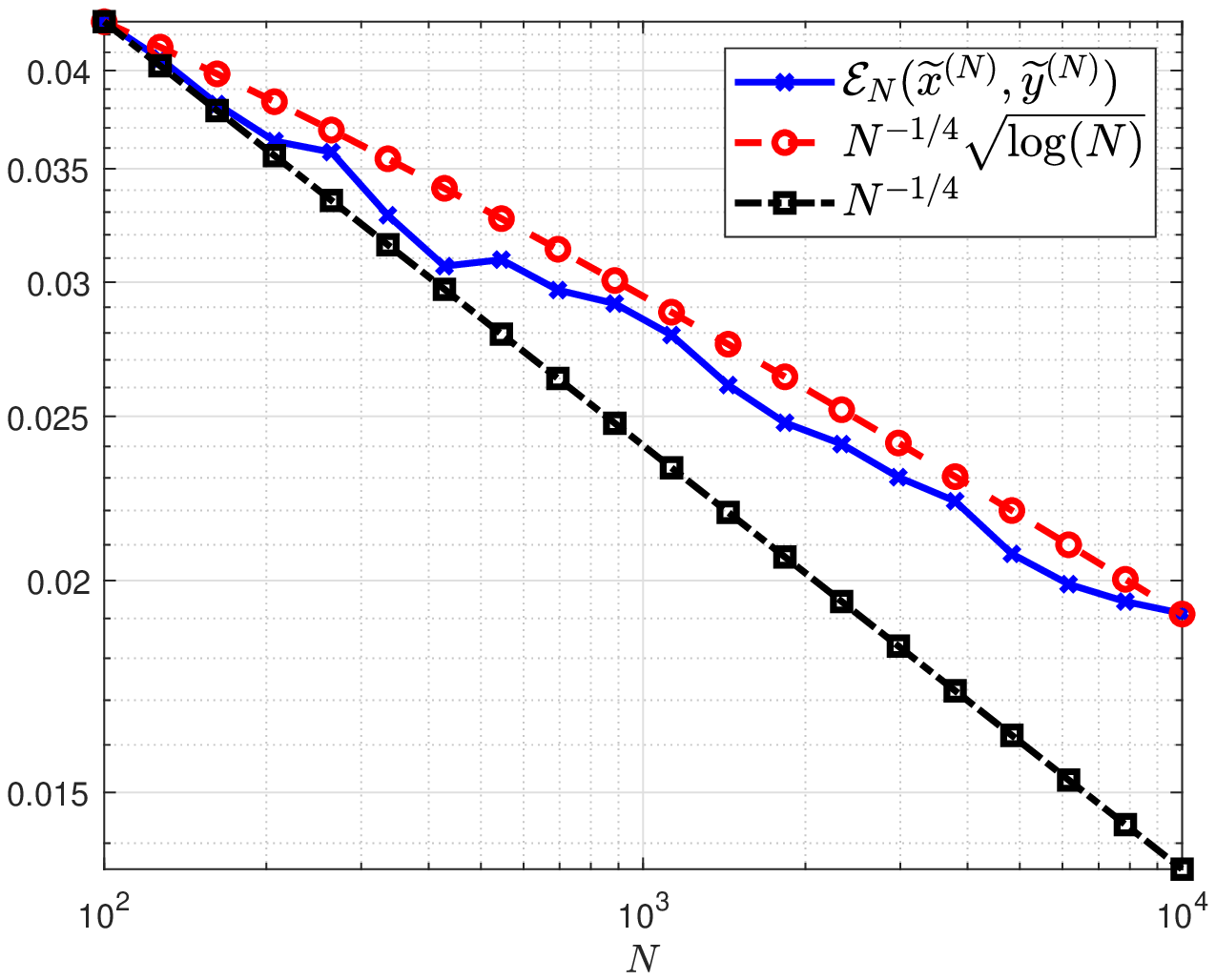} \label{fig:convergence of scaling factors rectangular matrix}
    }
    }
    \caption
    {Empirical values of $\mathcal{E}_N(\widetilde{\mathbf{x}}^{(N)},\widetilde{\mathbf{y}}^{(N)})$ as function of $N$ in log-log scale, compared to the corresponding bounds from Corollary~\ref{cor:relative convergence of scaling factors}. Figure~\ref{fig:convergence of scaling factors doubly stochastic} corresponds to $M_N = N$ and $r_i^{(N)} = c_j^{(N)} = 1$ for all $i\in[M]$ and $j\in[N]$. Figure~\ref{fig:convergence of scaling factors rectangular matrix and random sums} corresponds to $M_N = 3N$, and $\{r_i^{(N)}\}, \{c_j^{(N)}\}$ sampled independently and uniformly from $[0.1,1]$ and normalized to sum to $1$. Figure~\ref{fig:convergence of scaling factors rectangular matrix} corresponds to $M_N = 10\sqrt{N}$ and $r_i^{(N)} = N, \; c_j^{(N)} = M_N$ for all $i\in[M]$ and $j\in[N]$. } \label{fig:convergence of scaling factors}
    \end{figure} 
    
    Figure~\ref{fig:convergence of scaled matrix} shows the behavior of the empirical error $\Vert \widetilde{P}^{(N)} - P^{(N)} \Vert_2 / \Vert P^{(N)} \Vert_2$ as a function of $N$ for the same scenarios as in Figure~\ref{fig:convergence of scaling factors}. For these scenarios, the rates that govern the bounds on $\Vert \widetilde{P}^{(N)} - P^{(N)} \Vert_2 / \Vert P^{(N)} \Vert_2$ according to Corollary~\ref{cor:relative convergence of scaled matrix in operator norm} are the same as those for $\mathcal{E}_N(\widetilde{\mathbf{x}}^{(N)},\widetilde{\mathbf{y}}^{(N)})$ from Corollary~\ref{cor:relative convergence of scaling factors} (described previously in the context of Figure~\ref{fig:convergence of scaling factors}). In the scenario where $M_N$ grows proportionally to $N$, it is evident from Figures~\ref{fig:convergence of scaled matrix doubly stochastic} and~\ref{fig:convergence of scaled matrix rectangular matrix and random sums} that the bound in Corollary~\ref{cor:relative convergence of scaled matrix in operator norm} is tight up to the factor $\sqrt{\log(N)}$, suggesting that the factor $\sqrt{\log(N)}$ is probably not required in the bound on $\Vert \widetilde{P}^{(N)} - P^{(N)} \Vert_2 / \Vert P^{(N)} \Vert_2$ (in contrast to the bound on $\mathcal{E}_N(\widetilde{\mathbf{x}}^{(N)},\widetilde{\mathbf{y}}^{(N)})$ depicted in Figure~\ref{fig:convergence of scaling factors}). In the scenario where $M_N$ grows disproportionately to $N$, Figure~\ref{fig:convergence of scaling factors rectangular matrix} empirically suggests that the bound in Corollary~\ref{cor:relative convergence of scaled matrix in operator norm} can be improved by a factor of ${\log(N)}$, which would bring the rate in this case to be $\mathcal{O}_{\text{w.h.p}}(N^{-1/4}/\log N)$. Overall, these experiments suggest that the bound in Corollary~\ref{cor:relative convergence of scaled matrix in operator norm} describes the correct behavior of $\Vert \widetilde{P}^{(N)} - P^{(N)} \Vert_2 / \Vert P^{(N)} \Vert_2$ with $N$ up to poly-logarithmic factors.  

\begin{figure} 
  \centering
  	{
  	\subfloat[$M_N = N$, $r_i^{(N)} = c_j^{(N)} = 1$]  
  	{
    \includegraphics[width=0.3\textwidth]{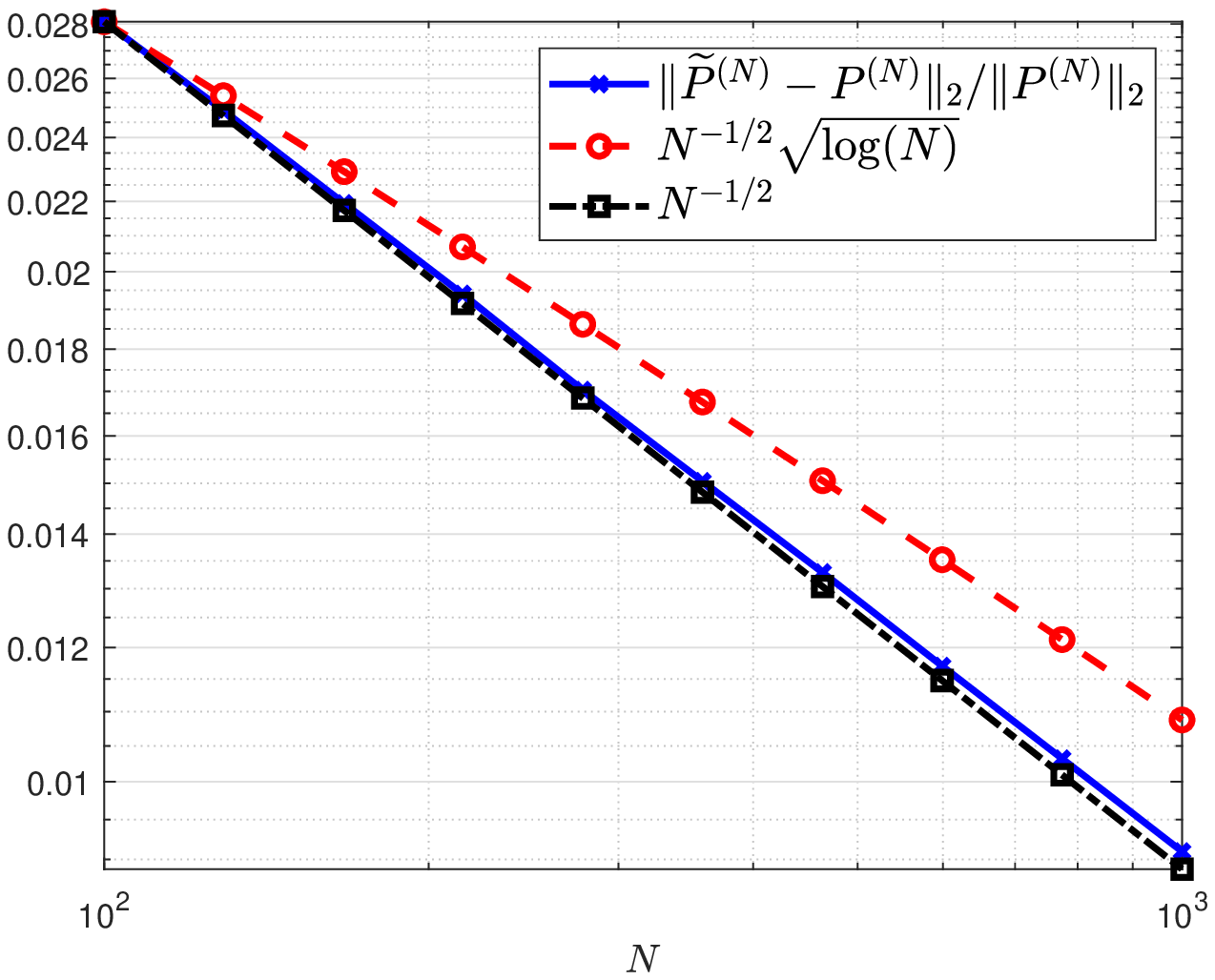} \label{fig:convergence of scaled matrix doubly stochastic}
    }
    \subfloat[$M_N = 3N$, $r_i,c_j$ are random]  
  	{
    \includegraphics[width=0.3\textwidth]{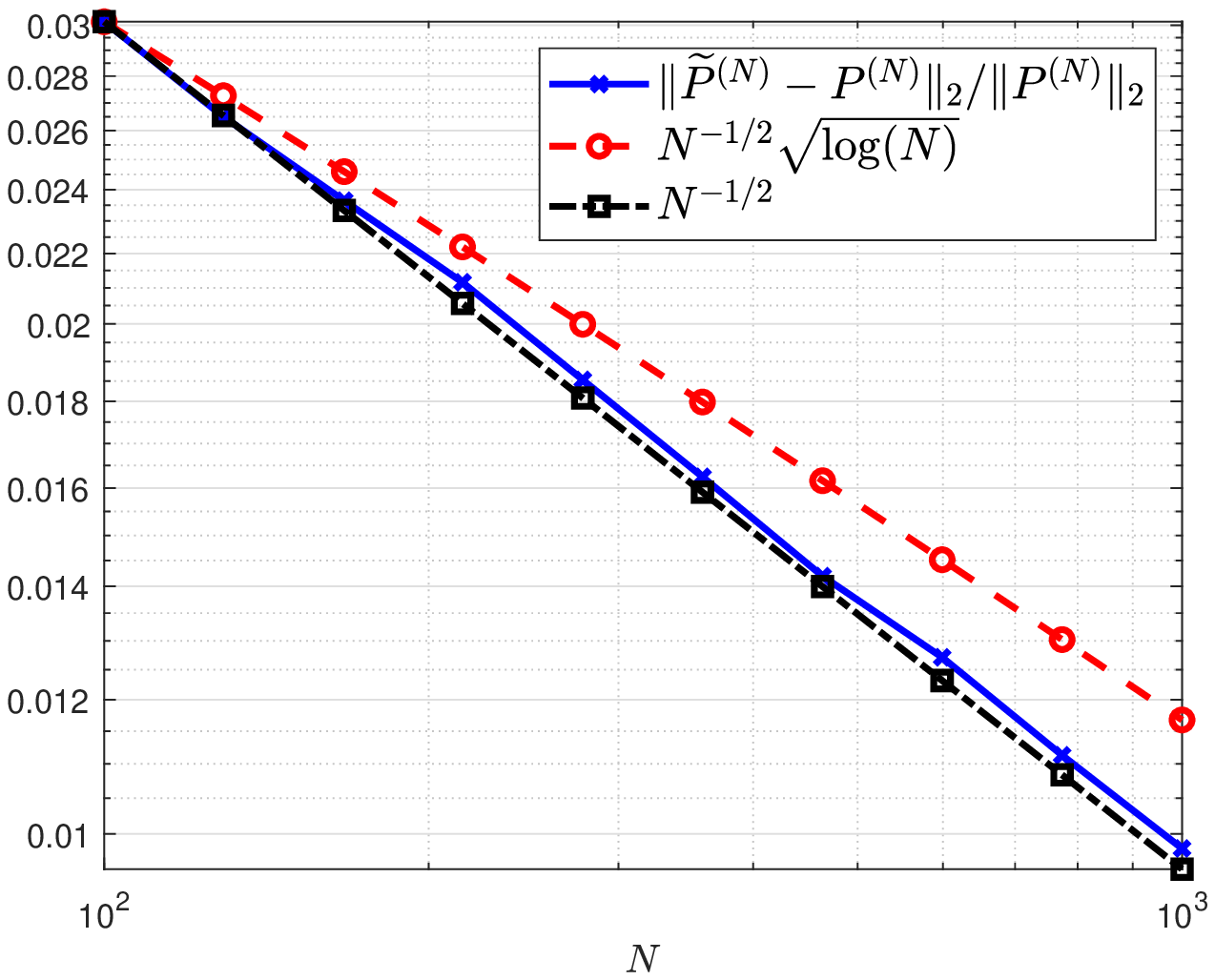} \label{fig:convergence of scaled matrix rectangular matrix and random sums}
    }
    \subfloat[$M_N=10\sqrt{N}$, $r_i=N$, $c_j=M_N$]  
  	{
    \includegraphics[width=0.3\textwidth]{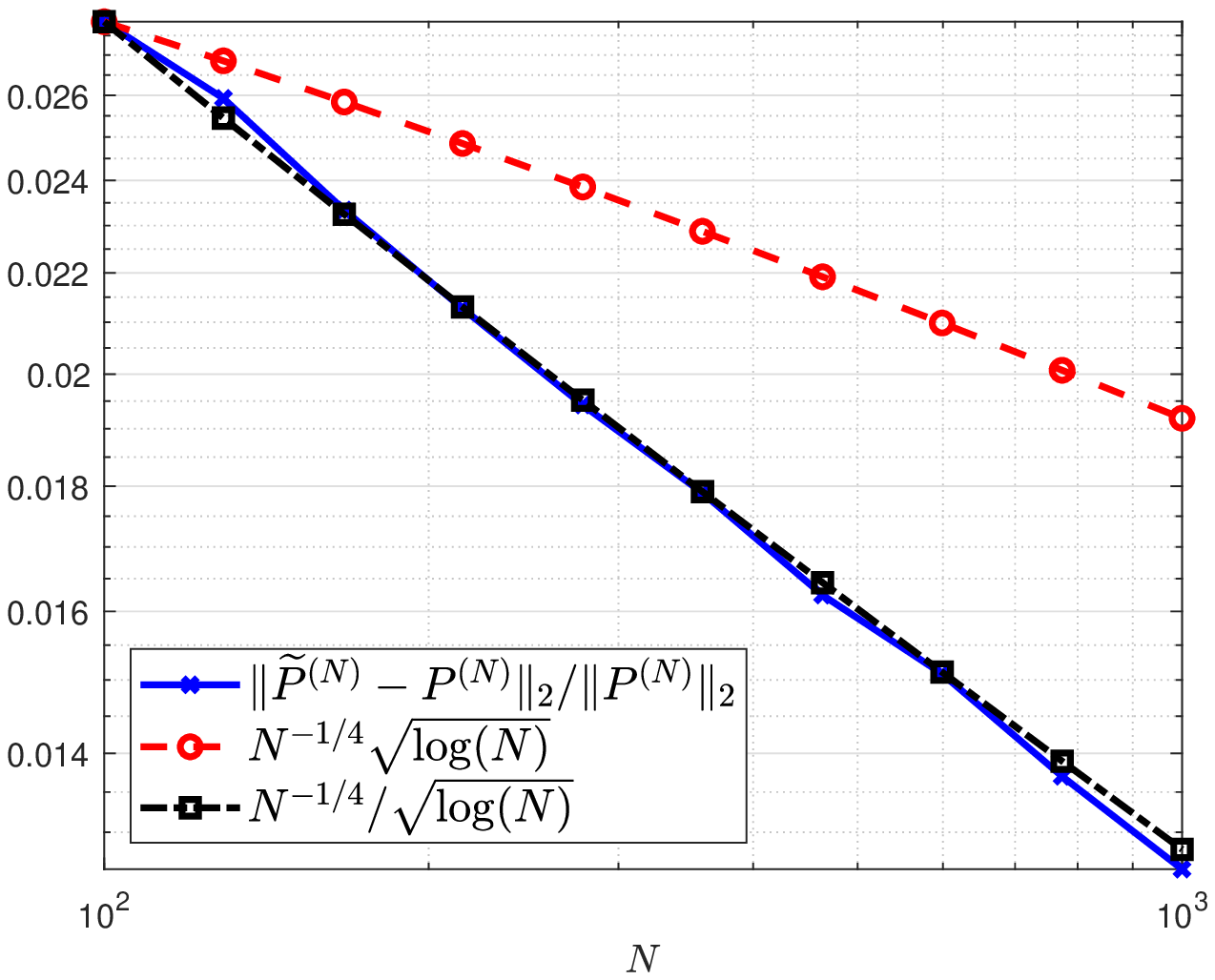} \label{fig:convergence of scaled matrix rectangular matrix}
    }
    }
    \caption
    {Empirical values of $\Vert \widetilde{P}^{(N)} - P^{(N)} \Vert_2 / \Vert P^{(N)} \Vert_2$ as a function of $N$ in log-log scale for the same settings as in Figure~\ref{fig:convergence of scaling factors}, compared to the corresponding bounds from Corollary~\ref{cor:relative convergence of scaled matrix in operator norm}.} \label{fig:convergence of scaled matrix}
    \end{figure} 

\section{Proofs}

\subsection{Proof of Lemma~\ref{lem:boundedness of scaling factors}} \label{sec:proof of boundednes lemma}
Theorem~\ref{thm:existence and uniquness of matrix scaling} guarantees the existence of a pair $(\mathbf{x}^{'},\mathbf{y}^{'})$ of scaling factors of $A$ and states that all possible pairs of scaling factors of $A$ must be of the form $(\alpha \mathbf{x}^{'}, \alpha^{-1} \mathbf{y}^{'})$, for $\alpha>0$. Note that setting $\Vert \alpha \mathbf{x}^{'} \Vert_1 = \Vert \alpha^{-1} \mathbf{y}^{'} \Vert_1$ determines $\alpha$ uniquely, and consequently, there exists a unique pair $(\mathbf{x},\mathbf{y})$ such that $\Vert \mathbf{x} \Vert_1 = \Vert \mathbf{y} \Vert_1$. According to~\eqref{eq:equations for x and y in matrix scaling definition} we have
\begin{equation}
x_i = \frac{r_i}{\sum_{j=1}^N A_{i,j} y_j}, \qquad \qquad
y_j = \frac{c_j}{\sum_{i=1}^M A_{i,j} x_i},
\end{equation}
and since $a \leq A_{i,j} \leq b$ for all $i,j$, it follows that
\begin{equation}
    \frac{r_i}{b \sum_{j=1}^N y_j} \leq x_i \leq \frac{r_i}{a \sum_{j=1}^N y_j}, \qquad \qquad
    \frac{c_j}{b \sum_{i=1}^M x_i} \leq y_j \leq \frac{c_j}{a \sum_{i=1}^M x_i}, \label{x_i and y_j preliminary inequalities}
\end{equation}
for all $i\in [M]$, $j\in [N]$.
Summing the inequalities for $x_i$ in~\eqref{x_i and y_j preliminary inequalities} over $i=1,\ldots,M$, and using $\sum_{i=1}^M x_i = \sum_{j=1}^N y_j$ together with $\sum_{i=1}^M r_i = s$, gives
\begin{equation}
    \sqrt{\frac{s}{b}} \leq \sum_{j=1}^N y_{j} = \sum_{i=1}^M x_i \leq \sqrt{\frac{{s}}{a}}.
\end{equation}
Lastly, plugging the above back into~\eqref{x_i and y_j preliminary inequalities} gives the required result.

\subsection{Lemmas supporting the proof of Theorem~\ref{thm:deviation of scaling factors}}\label{sec:supporting lemmas}
The first step in proving Theorem~\ref{thm:deviation of scaling factors} is to make use of Lemma~\ref{lem:boundedness of scaling factors} together with Hoeffding's inequality~\cite{hoeffding1994probability} to provide a concentration inequality for the sums of the rows and of the columns of $D(\mathbf{x}) \widetilde{A} D(\mathbf{y})$ around $\mathbf{r}$ and $\mathbf{c}$, respectively, where $(\mathbf{x},\mathbf{y})$ is any pair of scaling factors of $A$. This is the subject of the following lemma.

\begin{lem}[Concentration of row and column sums] \label{lem:deviation from scaling}
Suppose that $\{\widetilde{A}_{i,j}\}_{i=1}^M$ are independent for each $j\in [N]$, and $\{\widetilde{A}_{i,j}\}_{j=1}^N$ are independent for each $i\in [M]$. Furthermore, suppose that $\widetilde{A}_{i,j} \in [a_{i,j},b_{i,j}]$ a.s. for all $i\in [M]$ and $j\in [N]$, and denote $a = \min_{i,j} a_{i,j}$, $b = \max_{i,j}b_{i,j}$. Then, for any pair of vectors $(\mathbf{x},\mathbf{y})$ that scales $A=\mathbb{E}[\widetilde{A}]$ to row sums $\mathbf{r}$ can column sums $\mathbf{c}$, we have
    \begin{align}
    &\operatorname{Pr}\left\{ \left\vert \frac{1}{r_i} \sum_{j=1}^N  x_i \widetilde{A}_{i,j} y_j -1 \right\vert > \varepsilon \right\}  
    \leq 2 \operatorname{exp}\left( \frac{-2\varepsilon^2 s^2 }{ C^2 \sum_{j=1}^N c_j^2 (b_{i,j} - a_{i,j})^2}\right), \label{eq:row normalization random bound}\\
    &\operatorname{Pr}\left\{ \left\vert \frac{1}{c_j} \sum_{i=1}^M  x_i \widetilde{A}_{i,j} y_j -1 \right\vert > \varepsilon \right\}  
    \leq 2 \operatorname{exp}\left( \frac{-2\varepsilon^2 s^2 }{ C^2 \sum_{i=1}^M r_i^2 (b_{i,j} - a_{i,j})^2}\right), \label{eq:column normalization random bound}
\end{align}
for any $\varepsilon>0$ and all $i\in [M]$ and $j\in [N]$, where $C = b/a^2$.
\end{lem}
\begin{proof}
Observe that if $\widetilde{A}_{i,j} \in [a_{i,j}, b_{i,j}]$ a.s., then $A_{i,j} = \mathbb{E}[\widetilde{A}_{i,j}] \in [a_{i,j}, b_{i,j}]$. 
Therefore, using the fact that $\sum_j x_i {A}_{i,j} y_j = r_i$, Hoeffding's inequality~\cite{hoeffding1994probability} gives that
\begin{align}
    \operatorname{Pr}\left\{ \left\vert \frac{1}{r_i} \sum_{j=1}^N  x_i \widetilde{A}_{i,j} y_j -1 \right\vert > \varepsilon \right\} 
    &= \operatorname{Pr}\left\{ \left\vert \sum_{j=1}^N  x_i \widetilde{A}_{i,j} y_j - \sum_{j=1}^N x_i A_{i,j} y_j \right\vert > \varepsilon r_i \right\} \nonumber \\
    &\leq 2 \operatorname{exp}\left( \frac{-2\varepsilon^2 r_i^2 }{\sum_{j=1}^N x_i^2 y_j^2 (b_{i,j} - a_{i,j})^2 }\right), \label{eq:Hoeffding bound use 1}
\end{align}
for all $i\in [M]$.
Since we can always find a constant $\alpha>0$ such that $\Vert \alpha \mathbf{x} \Vert_1 = \Vert \alpha^{-1} \mathbf{y} \Vert_1$, Lemma~\ref{lem:boundedness of scaling factors} implies that for all $i\in [M]$ and $j\in[N]$,
\begin{equation}
    x_i y_j =  (\alpha {x_i}) (\alpha^{-1} y_j) \leq \overline{r}_i\overline{c}_j \frac{b}{a^2} = C \overline{r}_i\overline{c}_j.
\end{equation}
Applying the above inequality to~\eqref{eq:Hoeffding bound use 1}, we get
\begin{align}
    \operatorname{Pr}\left\{ \left\vert \frac{1}{r_i} \sum_{j=1}^N  x_i \widetilde{A}_{i,j} y_j -1 \right\vert > \varepsilon \right\}
    &\leq 2 \operatorname{exp}\left( \frac{-2\varepsilon^2 r_i^2 }{C^2 \overline{r_i}^2 \sum_{j=1}^N \overline{c_j}^2 (b_{i,j} - a_{i,j})^2}\right) \nonumber \\
    &= 2 \operatorname{exp}\left( \frac{-2\varepsilon^2 s^2 }{ C^2 \sum_{j=1}^N c_j^2 (b_{i,j} - a_{i,j})^2}\right), \label{eq: concentration inequality for row sums in proof}
\end{align}
for all $i\in [M]$.
Analogously to the derivation of~\eqref{eq: concentration inequality for row sums in proof}, by using $\sum_i x_i {A}_{i,j} y_j = c_j$ together with Hoeffding's inequality, one can verify that
\begin{align}
    \operatorname{Pr}\left\{ \left\vert \frac{1}{c_j} \sum_{i=1}^M  x_i \widetilde{A}_{i,j} y_j -1 \right\vert > \varepsilon \right\} 
    &=  \operatorname{Pr}\left\{ \left\vert \sum_{i=1}^M  x_i \widetilde{A}_{i,j} y_j - \sum_{i=1}^M x_i A_{i,j} y_j \right\vert > \varepsilon c_j \right\} \nonumber \\
    &\leq 2 \operatorname{exp}\left( \frac{-2\varepsilon^2 s^2 }{ C^2 \sum_{i=1}^M r_i^2 (b_{i,j} - a_{i,j})^2}\right),
\end{align}
for all $j\in [N]$
    
\end{proof}

We next prove that if $\widetilde{A}$ can be approximately scaled by a pair of vectors $(\mathbf{x},\mathbf{y})$, then there exists a pair of vectors $(\widetilde{\mathbf{x}},\widetilde{\mathbf{y}})$ that scales $\widetilde{A}$ and is also close to $(\mathbf{x},\mathbf{y})$. The proof relies on extending Sinkhorn's original proof of uniqueness of the scaling factors in~\cite{sinkhorn1967diagonal} to describe the stability of the scaling factors under approximate scaling. We note that $\widetilde{A}$ is not considered as random in the following Lemma.
\begin{lem} [Stability of scaling factors under approximate scaling] \label{lem:closeness of scaling factors}
Let $\widetilde{A}\in\mathbb{R}^{M\times N}$ be a positive matrix and denote $a = \min_{i,j} {\widetilde{A}}_{i,j}$, $b = \max_{i,j} \widetilde{A}_{i,j}$. Suppose that there exists $\varepsilon\in (0,1)$ and positive vectors $\mathbf{x} = [x_1,\ldots,x_M]$ and $\mathbf{y} = [y_1,\ldots,y_N]$, such that
\begin{equation}
\left\vert \frac{1}{c_j} \sum_{i=1}^M x_i \widetilde{A}_{i,j} y_j - 1 \right\vert \leq \varepsilon, \qquad \left\vert \frac{1}{r_i} \sum_{j=1}^N x_i \widetilde{A}_{i,j} y_j -1 \right\vert \leq \varepsilon, \label{eq:approximate scaling condition in lemma}
\end{equation}
for all $i\in [M]$ and $j\in [N]$. Then, $\widetilde{A}$ can be scaled to row sums $\mathbf{r}$ and column sums $\mathbf{c}$ by a pair $(\widetilde{\mathbf{x}},\widetilde{\mathbf{y}})$ that satisfies
\begin{align}
     \frac{\vert\widetilde{x}_i - x_i \vert}{x_i} &\leq \frac{\varepsilon}{1-\varepsilon} + \frac{4 \varepsilon s \sqrt{b} }{a^2 C_1^{3/2} C_2^{3/2} M \min_i {r}_i }, \\
    \frac{\vert\widetilde{y}_j - y_j \vert}{y_j} &\leq \frac{\varepsilon}{1-\varepsilon} + \frac{4 \varepsilon s \sqrt{b} }{a^2 C_1^{3/2} C_2^{3/2} N \min_j {c}_j },
\end{align}
for all $i\in [M]$ and $j\in [N]$, where $C_1 = min_{i} \{x_i/\overline{r}_i\}$ and $C_2 = \min_j \{y_j/\overline{c}_j\}$.
\end{lem}
\begin{proof}
Let $(\widetilde{\mathbf{x}},\widetilde{\mathbf{y}})$ be the unique pair of scaling factors of $\widetilde{A}$ with $\Vert \widetilde{\mathbf{x}} \Vert_1 = \Vert \widetilde{\mathbf{y}} \Vert_1$ (see Lemma~\ref{lem:boundedness of scaling factors}), and define
\begin{equation}
    \hat{P}_{i,j} = {x}_i \widetilde{A}_{i,j} {y}_j, \qquad \widetilde{P}_{i,j} = \widetilde{x}_i \widetilde{A}_{i,j} \widetilde{y}_j = u_i \hat{P}_{i,j} v_j, \qquad u_i = \frac{\widetilde{x}_i}{x_i}, \qquad v_j = \frac{\widetilde{y}_j}{y_j}, \label{eq:P,P_hat,u, and v def}
\end{equation}
for all $i\in [M]$ and $j\in [N]$. Observe that $\sum_i \widetilde{P}_{i,j} = c_j$, $\sum_j \widetilde{P}_{i,j} = r_i$, and
\begin{equation}
    \left\vert \frac{1}{c_j} \sum_{i=1}^M \hat{P}_{i,j} - 1 \right\vert \leq \varepsilon, \qquad\qquad \left\vert \frac{1}{r_i} \sum_{j=1}^N \hat{P}_{i,j} -1 \right\vert \leq \varepsilon, \label{eq:u_i P_ij v_j sum closeness bounds}
\end{equation}
for all $i\in[M]$, $j\in[N]$.
The proof of Lemma~\ref{lem:closeness of scaling factors} is based on a manipulation of~\eqref{eq:u_i P_ij v_j sum closeness bounds} using $\widetilde{P}$, $\hat{P}$, $\mathbf{u} = [u_1,\ldots,u_M]$, $\mathbf{v} = [v_1,\ldots,v_N]$, and their propertie. Since this manipulation is somewhat technical, in what follows we break it down into several steps.
\subsubsection{Deriving bounds on $\{u_i\}$ and $\{v_j\}$}
Using the first inequality in~\eqref{eq:u_i P_ij v_j sum closeness bounds} for $j = \operatorname{argmin}_k v_k$, we have
\begin{equation}
    1 = \frac{1}{c_j} \sum_{i=1}^M \widetilde{P}_{i,j} = \frac{1}{c_j}\sum_{i=1}^M u_i \hat{P}_{i,j} v_j \leq \min_j v_j \max_i u_i  \frac{1}{c_j} \sum_{i=1}^M \hat{P}_{i,j} \leq (1+\varepsilon) \min_j v_j \max_i u_i. \label{eq:min v max u lower bound}
\end{equation}
Similarly, using the second inequality in~\eqref{eq:u_i P_ij v_j sum closeness bounds} for $i = \operatorname{argmax}_\ell v_\ell$ gives
\begin{equation}
    1 = \frac{1}{r_i}\sum_{j=1}^N \widetilde{P}_{i,j} = \frac{1}{r_i}\sum_{j=1}^N u_i \hat{P}_{i,j} v_j \geq \max_i u_i \min_j v_j  \frac{1}{r_i} \sum_{j=1}^N \hat{P}_{i,j} \geq (1-\varepsilon) \max_i u_i \min_j v_j, \label{eq: max u min v upper bound}
\end{equation}
and by combining~\eqref{eq:min v max u lower bound} and~\eqref{eq: max u min v upper bound} we obtain
\begin{equation}
    \frac{1}{1+\varepsilon} \leq \max_i u_i \min_j v_j \leq \frac{1}{1-\varepsilon}. \label{eq:max u min v upper and lower bounds}
\end{equation}
Analogously to~\eqref{eq:min v max u lower bound} and~\eqref{eq: max u min v upper bound}, it is easy to verify that by using the first inequality in~\eqref{eq:u_i P_ij v_j sum closeness bounds} for $j = \operatorname{argmax}_k v_k$ and using the second inequality in~\eqref{eq:u_i P_ij v_j sum closeness bounds} for $i = \operatorname{argmin}_\ell v_\ell$, one gets 
\begin{equation}
    \frac{1}{1+\varepsilon} \leq \min_i u_i \max_j v_j \leq \frac{1}{1-\varepsilon}. \label{eq:min u max v upper and lower bounds}
\end{equation}
Note that~\eqref{eq:min u max v upper and lower bounds} can also be obtained directly from~\eqref{eq:max u min v upper and lower bounds} by a symmetry argument, that is, by considering~\eqref{eq:max u min v upper and lower bounds} in the setting when $\widetilde{A}$ is replaced with its transpose, thereby interchanging the roles of $\mathbf{u}$ and $\mathbf{v}$.  

In addition, according to Lemma~\ref{lem:boundedness of scaling factors} and using the fact that $x_i\geq C_1 \overline{r}_i$ and $y_j\geq C_2 \overline{c}_j$ (from the conditions in Lemma~\ref{lem:closeness of scaling factors}), it follows that for all $i\in [M]$ and $j\in [N]$
\begin{equation}
    u_i \leq \frac{\sqrt{b}}{a C_1}, \qquad \qquad v_j \leq \frac{\sqrt{b}}{a C_2}. \label{eq:u_i and v_j bounds}
\end{equation}

\subsubsection{Bounding $\max_j v_j - \min_j v_j$ and $\max_i u_i - \min_i u_i$}
Let us denote $\ell = \operatorname{argmax}_i u_i$. By the second inequality in~\eqref{eq:u_i P_ij v_j sum closeness bounds} together with~\eqref{eq:max u min v upper and lower bounds}, we can write
\begin{equation}
    1 = \frac{1}{r_\ell}\sum_{j=1}^N \widetilde{P}_{\ell,j} = \frac{1}{r_\ell}\sum_{j=1}^N u_\ell \hat{P}_{\ell,j} v_j \geq \frac{1}{(1+\varepsilon) r_\ell}\sum_{j=1}^N \hat{P}_{\ell,j} \frac{v_j}{\min_j v_j}, \label{eq:bounding max v - min v step 1}
\end{equation}
implying that
\begin{equation}
    \frac{1}{r_\ell} \sum_{j=1}^N \hat{P}_{\ell,j} \left( \frac{v_j}{\min_j v_j} -1 \right) \leq {1+\varepsilon} - \frac{1}{r_\ell}\sum_{j=1}^N \hat{P}_{\ell,j} \leq {2\varepsilon}. \label{eq:preliminary bound 1}
\end{equation}
Multiplying~\eqref{eq:preliminary bound 1} by $\min_j v_j /\min_j  \hat{P}_{\ell,j}$, it follows that
\begin{equation}
   \frac{1}{r_\ell} \sum_{j=1}^N ({v_j} - \min_j v_j ) \leq \frac{1}{r_\ell}  \sum_{j=1}^N \frac{\hat{P}_{\ell,j}}{\min_j \hat{P}_{\ell,j} } ({v_j} - \min_j v_j ) \leq 2\varepsilon \frac{\min_j v_j}{\min_j \hat{P}_{\ell,j} } \leq  \frac{2\varepsilon \min_j v_j}{a C_1 C_2 \overline{r}_\ell \min_j \overline{c}_j }, \label{eq:sum v_j -min v_j bound intermediate}
\end{equation}
where we used the definition of $\hat{P}$ together with the conditions in Lemma~\ref{lem:closeness of scaling factors}.
Multiplying~\eqref{eq:sum v_j -min v_j bound intermediate} by $r_\ell/N$ and employing the definitions of $\overline{r}_i$ and $\overline{c}_j$ (see~\eqref{eq:r_overline and c_overline def}) gives
\begin{equation}
    \frac{1}{N} \sum_{j=1}^N ({v_j} - \min_j v_j ) \leq \frac{2\varepsilon {s} \min_j v_j}{a C_1 C_2 N \min_j {c}_j } \leq \frac{2\varepsilon {s} \max_j v_j}{a C_1 C_2 N \min_j {c}_j }. \label{eq:bounding max v - min v step 2}
\end{equation}

We next provide a derivation analogous to~\eqref{eq:bounding max v - min v step 1}--\eqref{eq:bounding max v - min v step 2} to obtain a bound for $\frac{1}{N} \sum_{j=1}^N (\max_j v_j - v_j )$. 
Let us denote $t = \operatorname{argmin}_i u_i$. Using the second inequality in~\eqref{eq:u_i P_ij v_j sum closeness bounds} together with~\eqref{eq:min u max v upper and lower bounds}, we have
\begin{equation}
    1 = \frac{1}{r_t}\sum_{j=1}^N \widetilde{P}_{t,j} = \frac{1}{r_t}\sum_{j=1}^N u_t \hat{P}_{t,j} v_j \leq \frac{1}{(1-\varepsilon)r_t}\sum_{j=1}^N \hat{P}_{t,j} \frac{v_j}{\max_j v_j},
\end{equation}
and therefore
\begin{equation}
    \frac{1}{r_t} \sum_{j=1}^N \hat{P}_{t,j} \left( 1-\frac{v_j}{\max_j v_j}\right) \leq \frac{1}{r_t}\sum_{j=1}^N \hat{P}_{t,j} - (1-\varepsilon) \leq 2\varepsilon.
\end{equation}
Multiplying the above by $\max_j v_j /\min_j  \hat{P}_{t,j}$, it follows that
\begin{equation}
    \frac{1}{r_t} \sum_{j=1}^N (\max_j v_j  - v_j) \leq \frac{1}{r_t} \sum_{j=1}^N \frac{\hat{P}_{t,j}}{\min_j  \hat{P}_{t,j}} (\max_j v_j  - v_j) \leq 2\varepsilon \frac{\max_j v_j}{\min_j  \hat{P}_{t,j}} \leq \frac{2\varepsilon \max_j v_j}{a C_1 C_2 \overline{r}_t \min_j \overline{c}_j }.
\end{equation}
Furthermore, multiplying the above by $r_t/N$ and using the definitions of $\overline{r}_i$ and $\overline{c}_j$ (see~\eqref{eq:r_overline and c_overline def}), we get
\begin{equation}
   \frac{1}{N} \sum_{j=1}^N (\max_j v_j  - v_j) \leq 2\varepsilon \frac{\max_j v_j}{\min_j \hat{P}_{t,j} } \leq \frac{2\varepsilon s \max_j v_j }{a C_1 C_2 N \min_j {c}_j }. \label{eq:bounding max v - min v step 3}
\end{equation}

Lastly, summing~\eqref{eq:bounding max v - min v step 2} and~\eqref{eq:bounding max v - min v step 3} gives
\begin{equation}
    \max_j v_j - \min_j v_j \leq \frac{4 \varepsilon s \max_j v_j }{a C_1 C_2 N \min_j {c}_j }. \label{eq:max v_j - min v_j bound}
\end{equation}
It is easy to verify that by repeating the derivation of~\eqref{eq:bounding max v - min v step 1} --~\eqref{eq:max v_j - min v_j bound} analogously for $u_i$ instead of $v_j$, we get
\begin{equation}
    \max_i u_i - \min_i u_i \leq \frac{4 \varepsilon s \max_i u_i }{a C_1 C_2 M \min_i {r}_i }. \label{eq:max u_i - min u_i bound}
\end{equation}
We omit the full derivation for the sake of brevity. Note that~\eqref{eq:max u_i - min u_i bound} can also be obtained directly from~\eqref{eq:max v_j - min v_j bound} by a symmetry argument, namely by considering~\eqref{eq:max v_j - min v_j bound} in the setting where $\widetilde{A}$ is replaced with its transpose, so that $N$ is replaced with $M$, $\mathbf{c}$ is replaced with $\mathbf{r}$, and $\mathbf{v}$ is replaced with $\mathbf{u}$.

\subsubsection{Bounding $1 - \alpha u_i$ and $1 - \alpha^{-1} v_j$ for some $\alpha > 0$}
Observe that $|\tau - v_j| \leq \max_j v_j - \min_j v_j$ for any $\tau \in [\min_j v_j,\max_j v_j]$ and all $j\in[N]$. Taking $\tau$ as the geometric mean of $\max_j v_j$ and $\min_j v_j$, together with~\eqref{eq:max v_j - min v_j bound} gives
\begin{equation}
    \left\vert\sqrt{\max_j v_j \min_j v_j} -  v_j \right\vert \leq \frac{4 \varepsilon s \max_j v_j }{a C_1 C_2 N \min_j {c}_j }, \label{eq:max v min u geometric mean bound}
\end{equation}
for all $j\in [N]$. Multiplying both hand sides of~\eqref{eq:max v min u geometric mean bound} by $\alpha^{-1} = \sqrt{\max_i u_i / \max_j v_j}$ we get
\begin{equation}
    \left\vert\sqrt{\max_i u_i \min_j v_j} - \alpha^{-1} v_j \right\vert 
    \leq \frac{4 \varepsilon s \sqrt{\max_j v_j \max_i u_i} }{a C_1 C_2 N \min_j {c}_j } 
    \leq \frac{4 \varepsilon s \sqrt{b} }{a^2 C_1^{3/2} C_2^{3/2} N \min_j {c}_j },
\end{equation}
where we also used~\eqref{eq:u_i and v_j bounds} in the last inequality. According to~\eqref{eq:max u min v upper and lower bounds}, we have for all $\varepsilon\in (0,1)$ that
\begin{equation}
    1-\frac{\varepsilon}{1-\varepsilon} \leq {\frac{1}{1+\varepsilon}} \leq \sqrt{\frac{1}{1+\varepsilon}} \leq \sqrt{\max_i u_i \min_j v_j} \leq  \sqrt{\frac{1}{1-\varepsilon}} \leq \frac{1}{1-\varepsilon} = 1 + \frac{\varepsilon}{1-\varepsilon},
\end{equation}
which together with~\eqref{eq:max v min u geometric mean bound} implies that
\begin{equation}
    \left\vert 1 - \alpha^{-1} v_j \right\vert 
    \leq \frac{\varepsilon}{1-\varepsilon} + \frac{4 \varepsilon s \sqrt{b} }{a^2 C_1^{3/2} C_2^{3/2} N \min_j {c}_j }. \label{eq:1 - alpha_inv v_j bound}
\end{equation}

Analogously to~\eqref{eq:max v min u geometric mean bound}, from~\eqref{eq:max u_i - min u_i bound} we obtain
\begin{equation}
    \left\vert\sqrt{\max_i u_i \min_i u_i} -  u_i \right\vert \leq \frac{4 \varepsilon s \max_i u_i }{a C_1 C_2 M \min_i {r}_i }, \label{eq:max u min u geometric mean bound}
\end{equation}
for all $i\in [M]$. Multiplying both hand sides of~\eqref{eq:max u min u geometric mean bound} by $\alpha = \sqrt{ \max_j v_j / \max_i u_i }$ we get
\begin{equation}
    \left\vert\sqrt{\max_j v_j \min_i u_i} - \alpha u_i \right\vert 
    \leq \frac{4 \varepsilon s \sqrt{\max_j v_j \max_i u_i} }{a C_1 C_2 M \min_i {r}_i } 
    \leq \frac{4 \varepsilon s \sqrt{b} }{a^2 C_1^{3/2} C_2^{3/2} M \min_i {r}_i }.
\end{equation}
Consequently, using~\eqref{eq:min u max v upper and lower bounds}, and analogously to the derivation of~\eqref{eq:1 - alpha_inv v_j bound}, it follows that
\begin{equation}
    \left\vert 1 - \alpha u_i \right\vert 
    \leq \frac{\varepsilon}{1-\varepsilon} + \frac{4 \varepsilon s \sqrt{b} }{a^2 C_1^{3/2} C_2^{3/2} M \min_i {r}_i }, \label{eq:1 - alpha u_i bound}
\end{equation}
which together with the definition of $\mathbf{u}$ and $\mathbf{v}$ in~\eqref{eq:P,P_hat,u, and v def} concludes the proof (since $(\alpha \widetilde{\mathbf{x}},\alpha^{-1} \widetilde{\mathbf{y}})$ is a pair of scaling factors of $\widetilde{A}$).

\end{proof}

\subsection{Proof of Theorem~\ref{thm:deviation of scaling factors}} \label{sec:proof of main theorem}
According to Lemma~\ref{lem:boundedness of scaling factors} there exists a unique pair of positive vectors $({\mathbf{x}},{\mathbf{y}})$ satisfying $\Vert {\mathbf{x}}\Vert_1 = \Vert {\mathbf{y}}\Vert_1$ that scales ${A}$ to row sums $\mathbf{r}$ and column sums $\mathbf{c}$, with
\begin{equation}
    \frac{\sqrt{a}}{b} \leq \frac{{x_i}}{\overline{r}_i} \leq \frac{\sqrt{b}}{a}, \qquad\qquad \frac{\sqrt{a}}{b} \leq \frac{{y_j}}{\overline{c}_j} \leq \frac{\sqrt{b}}{a}, \label{eq:x_i and y_j bounds in proof}
\end{equation}
for all $i\in [M]$ and $j\in [N]$.
Additionally, Lemma~\ref{lem:deviation from scaling} together with the union bound (over all $i\in[M]$ and $j\in[N]$), asserts that with probability at least
\begin{equation}
    1 - 2 \sum_{i=1}^M \operatorname{exp}\left( - \frac{2 \varepsilon^2 s^2}{C^2 \sum_{j=1}^N c_j^2 (b_{i,j} - a_{i,j})^2 } \right) - 2 \sum_{j=1}^N \operatorname{exp}\left(-\frac{2 \varepsilon^2 s^2}{C^2 \sum_{i=1}^M r_i^2 (b_{i,j} - a_{i,j})^2}\right),
\end{equation}
we have 
\begin{equation}
\max_{j \in [N]}\left\vert \frac{1}{M} \sum_{i=1}^M {x}_i \widetilde{A}_{i,j} {y}_j - 1 \right\vert \leq \varepsilon, \qquad \text{and} \qquad \max_{i \in [M]}\left\vert \frac{1}{N} \sum_{j=1}^N {x}_i \widetilde{A}_{i,j} {y}_j -1 \right\vert \leq \varepsilon, \label{eq:approx scalability for all i,j}
\end{equation}
where $C = b/a^2$.
If~\eqref{eq:approx scalability for all i,j} holds, we apply Lemma~\ref{lem:closeness of scaling factors} with $C_1 = \min_i \{x_i / \overline{r}_i \} \geq \sqrt{a}/b$ and $C_2 = \min_j \{y_j / \overline{c}_j\} \geq \sqrt{a}/b$ (using~\eqref{eq:x_i and y_j bounds in proof}), which guarantees that $\widetilde{A}$ can be scaled to row sums $\mathbf{r}$ and column sums $\mathbf{c}$ by a pair $(\widetilde{\mathbf{x}},\widetilde{\mathbf{y}})$ that satisfies
\begin{align}
     \frac{\vert\widetilde{x}_i - x_i \vert}{x_i} &\leq \frac{\varepsilon}{1-\varepsilon} + \frac{4 \varepsilon s \sqrt{b} }{a^2 C_1^{3/2} C_2^{3/2} M \min_i {r}_i } \leq \varepsilon \left( 2 + \frac{4 s}{M \min_i {r}_i } \left( \frac{b}{a} \right)^{7/2} \right), \\
    \frac{\vert\widetilde{y}_j - y_j \vert}{y_j} &\leq \frac{\varepsilon}{1-\varepsilon} + \frac{4 \varepsilon s \sqrt{b} }{a^2 C_1^{3/2} C_2^{3/2} N \min_j {c}_j } \leq \varepsilon \left( 2 + \frac{4 s}{N \min_j {c}_j } \left( \frac{b}{a} \right)^{7/2} \right),
\end{align}
for all $i\in [M]$, $j\in [N]$, and any $\varepsilon \in (0,1/{2}]$.
Overall, taking $\varepsilon = \delta/{2}$, and using the fact that $s = \Vert \mathbf{r}\Vert_1 \geq M \min_i r_i$ and $s = \Vert \mathbf{c}\Vert_1 \geq N \min_j c_j$, asserts that that exists a pair $(\widetilde{\mathbf{x}},\widetilde{\mathbf{y}})$ that scales $\widetilde{A}$ to row sums $\mathbf{r}$ and column sums $\mathbf{c}$, such that for any $\delta \in (0,1]$
\begin{align}
     \frac{\vert\widetilde{x}_i - x_i \vert}{x_i} &\leq  \frac{ \delta s}{M \min_i {r}_i } \left( 1 + 2 \left(\frac{b}{a} \right)^{7/2} \right) , \\
    \frac{\vert\widetilde{y}_j - y_j \vert}{y_j} &\leq \frac{ \delta s}{N \min_j {c}_j } \left( 1 + 2 \left(\frac{b}{a} \right)^{7/2} \right),
\end{align}
for all $i\in [M]$ and $j\in[N]$, with probability at least
\begin{equation}
     1 - 2 \sum_{i=1}^M \operatorname{exp}\left( - \frac{ \delta^2 s^2}{2 C^2 \sum_{j=1}^N c_j^2 (b_{i,j} - a_{i,j})^2 } \right) - 2 \sum_{j=1}^N \operatorname{exp}\left(-\frac{ \delta^2 s^2}{2 C^2 \sum_{i=1}^M r_i^2 (b_{i,j} - a_{i,j})^2}\right).
\end{equation}
Therefore, using $b_{i,j} - a_{i,j} \leq \max_{i,j} \{ b_{i,j} - a_{i,j} \} = d$ proves Theorem~\ref{thm:deviation of scaling factors} with
\begin{equation}
    C_p = \sqrt{2}\left(\frac{b d}{a^2}\right), \qquad \qquad C_e = 1 + 2\left(\frac{b}{a} \right)^{7/2}. \label{eq:C_r and C_s expressions}
\end{equation}

\subsection{Proof of Theorem~\ref{thm:relative entry-wise convergence of scaling factors}} \label{sec:proof of theorem on convergence rate of scaling factors}
\sloppy We begin by considering indices $N$ for which $\rho_1^{(N)} \rho_2^{(N)}\sqrt{{\log (\max \{ M_N, N \}) } } \leq 1/\sqrt{2C_p}$, where $C_p$ is from Theorem~\ref{thm:deviation of scaling factors}. In this case, we apply Theorem~\ref{thm:deviation of scaling factors} using $\delta = \rho_1^{(N)} \sqrt{2 C_p {\log (\max \{ M_N, N \}) } }$, noting that $\delta \leq 1$ as required since $\rho_2^{(N)} \geq 1$ for all $N$. Consequently, for each such index $N$ there exists a pair of positive vectors $(\widetilde{\mathbf{x}}^{(N)},\widetilde{\mathbf{y}}^{(N)})$ that scales $\widetilde{A}^{(N)}$ to row sums $\mathbf{r}^{(N)}$ and column sums $\mathbf{c}^{(N)}$, such that with probability at least
\begin{align}
    &1 - 2 M_N \operatorname{exp}\left( -\frac{\delta^2 (s^{(N)})^2}{ C_p \Vert \mathbf{c}^{(N)} \Vert_2^2} \right) - 2 N \operatorname{exp}\left(-\frac{\delta^2 (s^{(N)})^2}{ C_p \Vert \mathbf{r}^{(N)} \Vert_2^2}\right) \nonumber \\
    &\geq 1 - 2 M \operatorname{exp}\left( -2 \log (\max \{ M_N, N \})  \right) - 2 N \operatorname{exp}\left(-2 \log (\max \{ M_N, N \}) \right) \nonumber \\
    &\geq 1 - \frac{4}{\min \{ M_N, N \}}, \label{eq: probability lower bound in proof of asymptotic convegrence of scaling factors}
\end{align}
we have for all $i\in[M_N]$ and $j\in[N]$,
\begin{align}
\begin{aligned}
        \frac{\vert \widetilde{x}^{(N)}_i - x^{(N)}_i \vert}{x^{(N)}_i} &\leq  
      \frac{ C_e \delta s^{(N)} }{M_N \min_i r^{(N)}_i} \leq C_c \rho_1^{(N)} \rho_2^{(N)}\sqrt{{2 C_p \log (\max \{ M_N, N \}) } }, \\
     \frac{\vert \widetilde{y}^{(N)}_j - y^{(N)}_j \vert}{y^{(N)}_j} &\leq   \frac{ C_e \delta s^{(N)} }{N \min_j {c}^{(N)}_j } \leq C_c \rho_1^{(N)} \rho_2^{(N)}\sqrt{{2 C_p \log (\max \{ M_N, N \}) } }. \label{eq:convergence rate upper bound 1}
\end{aligned}
\end{align}

Next, we consider indices $N$ for which $\rho_1^{(N)} \rho_2^{(N)}\sqrt{{\log (\max \{ M_N, N \}) } } > 1/\sqrt{2C_p}$. In this case, we first apply Lemma~\ref{lem:boundedness of scaling factors} to $\widetilde{A}^{(N)}$, which states there exists a pair of positive vectors $(\widetilde{\mathbf{x}}^{(N)},\widetilde{\mathbf{y}}^{(N)})$ that scales $\widetilde{A}^{(N)}$ to row sums $\mathbf{r}^{(N)}$ and column sums $\mathbf{c}^{(N)}$, such that
\begin{equation}
        \frac{\sqrt{a}}{b} \leq \frac{\widetilde{x}^{(N)}_i}{\overline{r}^{(N)}_i} \leq \frac{\sqrt{b}}{a}, \qquad \qquad  \frac{\sqrt{a}}{b} \leq \frac{\widetilde{y}^{(N)}_j}{\overline{c}^{(N)}_j} \leq \frac{\sqrt{b}}{a},
    \end{equation}
for all $i\in [M_N]$ and $j\in[N]$.
Combining the above inequalities with the analogous inequalities for $\mathbf{x}^{(N)}$ and $\mathbf{y}^{(N)}$ (that scale $A$ and satisfy $\Vert \mathbf{x}^{(N)} \Vert_1 = \Vert \mathbf{y}^{(N)} \Vert_1$) gives
\begin{equation}
    \frac{\widetilde{x}^{(N)}_i}{{x}^{(N)}_i} \leq \left(\frac{{b}}{a}\right)^{3/2}, \qquad \qquad  \frac{\widetilde{y}^{(N)}_j}{{y}^{(N)}_j} \leq \left(\frac{{b}}{a}\right)^{3/2},
\end{equation}
for all $i\in [M_N]$ and $j\in[N]$.
Therefore, we have that
\begin{align}
    \mathcal{E}_N(\widetilde{\mathbf{x}}^{(N)},\widetilde{\mathbf{y}}^{(N)}) 
    &= \max \left\{ \max_{i\in [M_N]} \left\vert \frac{\widetilde{x}^{(N)}_i}{{x}^{(N)}_i} - 1 \right\vert, \max_{j\in [N]} \left\vert \frac{\widetilde{y}^{(N)}_j}{{y}^{(N)}_j} - 1 \right\vert \right\} \leq \left(\frac{{b}}{a}\right)^{3/2} \nonumber \\
    &< \sqrt{2C_p}\left(\frac{{b}}{a}\right)^{3/2}  \rho_1^{(N)} \rho_2^{(N)}\sqrt{{\log (\max \{ M_N, N \}) } }. \label{eq:convergence rate upper bound 2}
\end{align}

Since the lower bound in the right hand-side of~\eqref{eq: probability lower bound in proof of asymptotic convegrence of scaling factors} converges to $1$ as $N\rightarrow \infty$, \eqref{eq:convergence rate upper bound 1} together with~\eqref{eq:convergence rate upper bound 2} provide the required result for all indices $N$.

\subsection{Proof of Theorem~\ref{thm:convergence of scaled matrix in operator norm}} \label{sec:proof of theorem on convergence in operator norm}
Let us define ${\eta}^{(N)} = \widetilde{\mathbf{x}}^{(N)} - {\mathbf{x}}^{(N)}  $, and $\xi^{(N)} = \widetilde{\mathbf{y}}^{(N)} - {\mathbf{y}}^{(N)}$. We can write
\begin{align}
    &\left\Vert D(\widetilde{\mathbf{x}}^{(N)} )\widetilde{A}^{(N)} D(\widetilde{\mathbf{y}}^{(N)} ) - D({\mathbf{x}}^{(N)} ) A^{(N) } D({\mathbf{y}}^{(N)} ) \right\Vert_2 
    \leq \Vert D({{x}}^{(N)} ) (\widetilde{A}^{(N)} - A^{(N) } ) D({{y}}^{(N)})\Vert_2 \nonumber \\
    &+ \Vert D(\eta^{(N)}) \widetilde{A} D({{y}}^{(N)}) \Vert_2 + \Vert  D({{x}}^{(N)} ) \widetilde{A} D(\xi^{(N)}) \Vert_2 +  \Vert D(\eta^{(N)}) \widetilde{A} D(\xi^{(N)}) \Vert_2. \label{eq:scaled matrix operator norm bound in proof}
\end{align}
We now bound the summands in the right-hand side of~\eqref{eq:scaled matrix operator norm bound in proof} one by one. Note that since $\{A_{i,j}^{(N)}\}_{i,j}$ are independent and are confined to the interval $[a,b]$ for each index $N$, the conditions in Theorem~\ref{thm:deviation of scaling factors} hold when replacing $\widetilde{A}$ with $\widetilde{A}^{(N)}$.
For the first summand in~\eqref{eq:scaled matrix operator norm bound in proof}, applying Lemma~\ref{lem:boundedness of scaling factors} to $A^{(N)}$ we have
\begin{align}
    &\Vert D({\mathbf{x}}^{(N)} ) (\widetilde{A}^{(N)} - A^{(N) } ) D({\mathbf{y}}^{(N)})\Vert_2 
    \leq \Vert D({\mathbf{x}}^{(N)} ) \Vert_2 \cdot \Vert \widetilde{A}^{(N)} - A^{(N) } \Vert_2 \cdot \Vert D({\mathbf{y}}^{(N)})\Vert_2 \nonumber \\
    &\leq \frac{b}{a^2} \frac{\max_i r_i^{(N)}}{\sqrt{s^{(N)}}} \Vert \widetilde{A}^{(N)} - A^{(N) } \Vert_2 \frac{\max_j c_j^{(N)}}{\sqrt{s^{(N)}}} = \frac{b \rho_3^{(N)} }{a^2 \sqrt{M_N N}} \Vert \widetilde{A}^{(N)} - A^{(N) } \Vert_2. \label{eq:first summand bound derivation}
\end{align}
Since $a \leq \widetilde{A}_{i,j}^{(N)} \leq b$, then also $a \leq A_{i,j}^{(N)}\leq b$, which implies that $a-b \leq \widetilde{A}_{i,j}^{(N)} - A_{i,j}^{(N)} \leq b-a$. Hence, $\{\widetilde{A}_{i,j}^{(N)} - A_{i,j}^{(N)}\}_{i,j}$ are independent, have mean zero, and  are bounded (and therefore sub-Gaussian). Applying Theorem 4.4.5 in~\cite{vershynin2018high} with $t = \sqrt{\log N}$ gives
\begin{equation}
    \Vert \widetilde{A}^{(N)} - A^{(N) } \Vert_2 = \mathcal{O}_{\text{w.h.p}} (\sqrt{N} + \sqrt{M_N} + \sqrt{\log N}). \label{eq:operator norm bound vershynin}
\end{equation}
Combining~\eqref{eq:operator norm bound vershynin} with~\eqref{eq:first summand bound derivation} asserts that
\begin{equation}
    \Vert D({\mathbf{x}}^{(N)} ) (\widetilde{A}^{(N)} - A^{(N) } ) D({\mathbf{y}}^{(N)})\Vert_2 = \mathcal{O}_{\text{w.h.p}} \left( \frac{\rho_3^{(N)}}{\sqrt{M_N N}} (\sqrt{N} + \sqrt{M_N} + \sqrt{\log N})\right). \label{eq:eq:first summand bound}
\end{equation}
Continuing, for the second summand in~\eqref{eq:scaled matrix operator norm bound in proof}, we have
\begin{align}
    &\Vert D(\eta^{(N)}) \widetilde{A} D({{y}}^{(N)}) \Vert_2 
    \leq \Vert D(\eta^{(N)}) \Vert_2 \cdot \Vert \widetilde{A} \Vert_2  \cdot \Vert D({{y}}^{(N)}) \Vert_2 
    \leq \Vert D(\eta^{(N)}) \Vert_2 \cdot \Vert \widetilde{A} \Vert_F   \frac{\max_j c_j^{(N)}}{\sqrt{s^{(N)}}} \nonumber  \\ 
    &\leq \max_{i} |\eta_i^{(N)}| \frac{b \sqrt{M_N N }\max_j c_j^{(N)}}{\sqrt{s^{(N)}}} 
    \leq \max_i \frac{| \widetilde{x}_i^{(N)} - x_i^{(N)} | }{x_i^{(N)}} \cdot \max_i x_i^{(N)} \cdot  \frac{b \sqrt{M_N N }\max_j c_j^{(N)}}{\sqrt{s^{(N)}}} \nonumber \\
    &\leq \max_i \frac{| \widetilde{x}_i^{(N)} - x_i^{(N)} | }{x_i^{(N)}} \cdot \frac{b^{3/2} \sqrt{M_N N } \max_i r_i^{(N)} \cdot \max_j c_j^{(N)}}{{a s^{(N)}}} \nonumber \\
    &= \mathcal{O}_{\text{w.h.p}}\left( \rho_1^{(N)} \rho_2^{(N)} \rho_3^{(N)} \sqrt{ \log (\max \{ M_N, N \}) } \right), \label{eq:second summand bound derivation}
\end{align}
where we used Lemma~\ref{lem:boundedness of scaling factors} (applied to $A^{(N)}$) and Theorem~\ref{thm:relative entry-wise convergence of scaling factors}. Analogously, it is easy to verify that the third summand in~\eqref{eq:scaled matrix operator norm bound in proof} admits the same bound as the second summand, namely
\begin{equation}
    \Vert  D({{x}}^{(N)} ) \widetilde{A} D(\xi^{(N)}) \Vert_2 = \mathcal{O}_{\text{w.h.p}}\left( \rho_1^{(N)} \rho_2^{(N)} \rho_3^{(N)} \sqrt{\log (\max \{ M_N, N \}) } \right). \label{eq:third summand bound derivation}
\end{equation}
For the fourth summand in~\eqref{eq:scaled matrix operator norm bound in proof}, we write
\begin{align}
    &\Vert D(\eta^{(N)}) \widetilde{A} D(\xi^{(N)}) \Vert_2 
    \leq \Vert D(\eta^{(N)}) \Vert_2 \cdot \Vert \widetilde{A} \Vert_2 \cdot \Vert D(\xi^{(N)}) \Vert_2 
    \leq \max_i |\eta_i^{(N)}| \cdot \Vert \widetilde{A} \Vert_F \cdot \max_j |\xi_j^{(N)}| \nonumber \\
    &\leq \max_i \frac{| \widetilde{x}_i^{(N)} - x_i^{(N)} | }{x_i^{(N)}} \cdot \max_i x_i^{(N)} \cdot \sqrt{M_N N} \cdot \max_j \frac{| \widetilde{y}_j^{(N)} - y_j^{(N)} | }{y_j^{(N)}} \cdot \max_j y_j^{(N)} \nonumber \\
    &\leq \max_i \frac{| \widetilde{x}_i^{(N)} - x_i^{(N)} | }{x_i^{(N)}} \cdot \max_j \frac{| \widetilde{y}_j^{(N)} - y_j^{(N)} | }{y_j^{(N)}}  \cdot \frac{b \sqrt{M_N} \max_i r_i^{(N)} \cdot \sqrt{N} \max_j c_j^{(N)} }{a^2 s^{(N)}  } \nonumber \\
    &= \mathcal{O}_{\text{w.h.p}} \left( \left(\rho_1^{(N)} \rho_2^{(N)} \sqrt{\log (\max \{ M_N, N \})} \right)^2 \rho_3^{(N)}  \right), \label{eq:first summand bound derivation 4}
\end{align}
where we again used Lemma~\ref{lem:boundedness of scaling factors} (applied to $A^{(N)}$) and Theorem~\ref{thm:relative entry-wise convergence of scaling factors}.

Observe that $s^{(N)} = \Vert \mathbf{r}^{(N)} \Vert_1 \leq \sqrt{M_N} \Vert \mathbf{r}^{(N)} \Vert_2$, and $s^{(N)} = \Vert \mathbf{c}^{(N)} \Vert_1 \leq \sqrt{N} \Vert \mathbf{c}^{(N)} \Vert_2$. Therefore, 
\begin{equation}
    \rho_1^{(N)} = \max \left\{\frac{\Vert \mathbf{r}^{(N)} \Vert_2}{s^{(N)}}, \frac{\Vert \mathbf{c}^{(N)} \Vert_2}{s^{(N)}} \right\} \geq \max \{ \frac{1}{\sqrt{M_N}}, \frac{1}{\sqrt{N}} \}.
\end{equation}
Using the above together with the fact that $\rho_2^{(N)} \geq 1$, it follows that
\begin{align}
    &\rho_1^{(N)} \rho_2^{(N)} \rho_3^{(N)} \sqrt{{\log (\max \{ M_N, N \}) } } 
    \geq \rho_3^{(N)} \max \{ \frac{1}{\sqrt{M_N}}, \frac{1}{\sqrt{N}} \} \sqrt{{\log (\max \{ M_N, N \}) } }  \nonumber \\
    &= \frac{\rho_3^{(N)}}{\sqrt{M_N N}} \max \{\sqrt{M_N},\sqrt{N} \}  \sqrt{ \log (\max \{ M_N, N \}) } \nonumber \\
    &\geq \frac{\rho_3^{(N)}}{2 \sqrt{M_N N}} ( \sqrt{M_N} + \sqrt{N} ) \sqrt{\log N} \geq \frac{\rho_3^{(N)}}{4 \sqrt{M_N N}} ( \sqrt{M_N} + \sqrt{N} + \sqrt{\log N} ),
\end{align}
for all sufficiently large indices $N$. Applying the above inequality to~\eqref{eq:eq:first summand bound} we obtain
\begin{equation}
    \Vert D({\mathbf{x}}^{(N)} ) (\widetilde{A}^{(N)} - A^{(N) } ) D({\mathbf{y}}^{(N)})\Vert_2 = \mathcal{O}_{\text{w.h.p}}\left( \rho_1^{(N)} \rho_2^{(N)} \rho_3^{(N)} \sqrt{{\log (\max \{ M_N, N \}) } } \right). \label{eq:first summand bound derivation 2}
\end{equation}

We first consider indices $N$ for which $ \rho_1^{(N)} \rho_2^{(N)} \sqrt{{ \log (\max \{ M_N, N \}) } } \leq 1$. In this case, by~\eqref{eq:first summand bound derivation 4} we have
\begin{equation}
    \Vert D(\eta^{(N)}) \widetilde{A} D(\xi^{(N)}) \Vert_2 
    = \mathcal{O}_{\text{w.h.p}}\left( \rho_1^{(N)} \rho_2^{(N)} \rho_3^{(N)} \sqrt{{ \log (\max \{ M_N, N \}) } } \right), \label{eq:fourth summand bound}
\end{equation}
and plugging~\eqref{eq:first summand bound derivation 2},~\eqref{eq:second summand bound derivation},~\eqref{eq:third summand bound derivation}, and~\eqref{eq:fourth summand bound} into~\eqref{eq:scaled matrix operator norm bound in proof} gives that
\begin{equation}
    \left\Vert D(\widetilde{\mathbf{x}}^{(N)} )\widetilde{A}^{(N)} D(\widetilde{\mathbf{y}}^{(N)} ) - D({\mathbf{x}}^{(N)} ) A^{(N) } D({\mathbf{y}}^{(N)} ) \right\Vert_2 = \mathcal{O}_{\text{w.h.p}}\left( \rho_1^{(N)} \rho_2^{(N)} \rho_3^{(N)} \sqrt{{ \log (\max \{ M_N, N \}) } } \right). \label{eq:spectral error small upper bound}
\end{equation}

Next, we consider indices $N$ for which $ \rho_1^{(N)} \rho_2^{(N)} \sqrt{{ \log (\max \{ M_N, N \}) } } > 1$. In this case, we write
\begin{multline}
    \left\Vert D(\widetilde{\mathbf{x}}^{(N)} )\widetilde{A}^{(N)} D(\widetilde{\mathbf{y}}^{(N)} ) - D({\mathbf{x}}^{(N)} ) A^{(N) } D({\mathbf{y}}^{(N)} ) \right\Vert_2 \\
    \leq \left\Vert D(\widetilde{\mathbf{x}}^{(N)} )\widetilde{A}^{(N)} D(\widetilde{\mathbf{y}}^{(N)} ) \right\Vert_F  +  \left\Vert D({\mathbf{x}}^{(N)} ) A^{(N) } D({\mathbf{y}}^{(N)} ) \right\Vert_F. \label{eq:spectral error large bound}
\end{multline}
By applying Lemma~\ref{lem:boundedness of scaling factors} to $A^{(N)}$ and $\widetilde{A}^{(N)}$, we have that
\begin{equation}
    \widetilde{x}_i^{(N)} \widetilde{A}^{(N)} \widetilde{y}_j^{(N)} \leq \left( \frac{b}{a} \right)^2 \frac{ r_i^{(N)} c_j^{(N)} }{ s^{(N)} }, \qquad \text{and} \qquad {x}_i^{(N)} {A}^{(N)} {y}_j^{(N)} \leq \left( \frac{b}{a} \right)^2 \frac{ r_i^{(N)} c_j^{(N)} }{ s^{(N)} }, \label{eq:scaled matrices upper bound}
\end{equation}
for all $i\in[M_N]$ and $j\in[N]$. Therefore, Combining~\eqref{eq:scaled matrices upper bound} with~\eqref{eq:spectral error large bound} implies
\begin{align}
    &\left\Vert D(\widetilde{\mathbf{x}}^{(N)} )\widetilde{A}^{(N)} D(\widetilde{\mathbf{y}}^{(N)} ) - D({\mathbf{x}}^{(N)} ) A^{(N) } D({\mathbf{y}}^{(N)} ) \right\Vert_2
    \leq 2 \left( \frac{b}{a} \right)^2 \frac{\Vert \mathbf{r}^{(N)} \Vert_2 \cdot \Vert \mathbf{c}^{(N)} \Vert_2 }{s^{(N)}} \nonumber \\ 
    &\leq 2 \left( \frac{b}{a} \right)^2 \frac{\sqrt{M_N} \max_i r_i^{(N)} \cdot \sqrt{N} \max_j c_j^{(N)} }{s^{(N)}} 
    = 2 \left( \frac{b}{a} \right)^2 \rho_3^{(N)} \nonumber \\
    &< 2 \left( \frac{b}{a} \right)^2 \rho_1^{(N)} \rho_2^{(N)} \rho_3^{(N)} \sqrt{{ \log (\max \{ M_N, N \}) } } \nonumber \\ 
    &= \mathcal{O}_{\text{w.h.p}} \left(\rho_1^{(N)} \rho_2^{(N)} \rho_3^{(N)} \sqrt{{ \log (\max \{ M_N, N \}) } }\right),
    \label{eq:spectral error large bound 2}
\end{align}
where we used the fact that $ \rho_1^{(N)} \rho_2^{(N)} \sqrt{{ \log (\max \{ M_N, N \}) } } > 1$ for the considered indices $N$.

\section{Acknowledgements}
The author would like to thank Thomas Zhang, Yuval Kluger, and Dan Kluger for their useful comments and suggestions. This research was supported by the National Institute of Health, grant numbers R01GM131642 and UM1DA051410.

\bibliographystyle{plain}
\bibliography{mybib}

\end{document}